\documentclass[11pt,reqno]{amsart}

\usepackage{amsfonts, amsthm, amsmath}
\allowdisplaybreaks[4]

\usepackage{rotating}

\usepackage{tikz}
\usetikzlibrary{arrows.meta,positioning}

\usepackage{graphics}

\usepackage{amssymb}

\usepackage{mathtools}

\usepackage{booktabs}

\usepackage{circledsteps}

\usepackage{amscd}

\usepackage[latin2]{inputenc}

\usepackage{t1enc}

\usepackage[mathscr]{eucal}

\usepackage{indentfirst}

\usepackage{graphicx}

\usepackage{pict2e}

\usepackage{mathrsfs}

\usepackage{enumerate}

\usepackage[pagebackref]{hyperref}
\hypersetup{colorlinks=true}
\usepackage{color}
\usepackage{epic}
\usepackage{framed}

\numberwithin{equation}{section}
\theoremstyle{plain}

\newtheorem{theorem}{Theorem}[section]

\newtheorem{lemma}[theorem]{Lemma}

\newtheorem{corollary}[theorem]{Corollary}

\theoremstyle{definition}

\newtheorem{Def}[theorem]{Definition}

\newtheorem{example}[theorem]{Example}

\newtheorem{remark}[theorem]{Remark}

\newtheorem{?}[theorem]{Problem}

\newtheoremstyle{named}{}{}{\itshape}{}{\bfseries}{.}{.5em}{#1\thmnote{ #3}}
\theoremstyle{named}

\newcommand{\orcid}[1]{\href{https://orcid.org/#1}{\textsuperscript{\textregistered}}}

\newcommand{\f}[1]{\ifthenelse{\equal{#1}{1}}{(q;q)_\infty}{(q^{#1};q^{#1})_{\infty}}}

\def\bN{\mathbb{N}}

\def\sP{\mathscr{P}}
\def\sD{\mathscr{D}}
\def\sA{\mathscr{A}}
\def\sB{\mathscr{B}}
\def\sR{\mathscr{R}}
\def\o{\mathrm{o}}

\def\sE{\mathscr{E}}

\def\sn{\mathrm{sn}}
\def\pr{\mathrm{pr}}

\def\ri{\rightarrow}
\def\xri{\xrightarrow}
\def\la{\lambda}
\def\al{\alpha}
\def\ep{\epsilon}
\def\be{\beta}
\def\si{\sigma}
\def\ga{\gamma}

\def\hla{\hat{\lambda}}
\def\hmu{\hat{\mu}}
\def\I{\mathrm{I}}
\def\II{\mathrm{II}}
\def\III{\mathrm{III}}
\def\IV{\mathrm{IV}}
\def\V{\mathrm{V}}
\def\VI{\mathrm{VI}}
\def\VII{\mathrm{VII}}

\begin{document}
\title[Combinatorial decomp. of SP and multi-sum R-R type identities]{Combinatorial decompositions of strict partitions and multi-sum Rogers-Ramanujan type identities}

\author[H. Li]{Haijun Li
}
\address[Haijun Li]{College of Mathematics and Statistics, Chongqing University, Chongqing 401331, P.R. China}
\email{lihaijun@cqu.edu.cn; lihaijune@163.com}

\date{\today}

\begin{abstract}
While multi-sum Rogers-Ramanujan type identities have been extensively investigated through Bailey pairs, the constant term method, complex contour integration, and related techniques, many still lack an intrinsic combinatorial understanding. In this paper, we establish combinatorial interpretations for several families of triple-sum identities via structural decompositions of strict partitions. For one of these families, we extend the underlying combinatorial framework to obtain a general $(m+1)$-fold multi-sum identity. Leveraging these partition-theoretic interpretations, we construct explicit bijective proofs for all the corresponding identities. As an application, this constructive approach naturally yields some parameterized generalizations.
\end{abstract}


\maketitle
\section{Introduction}\label{sec:intro}

Throughout this paper, we adopt the standard $q$-series notation (see, for example,~Gasper and Rahman~\cite{GR04}):
\begin{align*}
&(a; q)_n := \prod_{k=0}^{n-1} (1 - aq^k),\quad  (a; q)_\infty := \prod_{k=0}^{\infty} (1 - aq^k), \quad |q| < 1,\\
&(a_1, \ldots, a_m; q)_n := (a_1; q)_n \cdots (a_m; q)_n, \qquad n \in \mathbb{N}_0 \cup \{\infty\},
\end{align*}
where $\mathbb{N}$ denotes the set of positive integers and $\mathbb{N}_0 := \mathbb{N} \cup \{0\}$.

Rogers-Ramanujan type identities are sum-to-product identities expressing a $q$-hypergeometric series, often in the form of a multisum, as an infinite product (or a linear combination of infinite products). The prototypical examples are the celebrated Rogers-Ramanujan identities:
\begin{align}
\sum_{n\geq 0}\frac{q^{n^2}}{(q; q)_n}=\frac{1}{(q, q^4; q^5)_{\infty}},\quad \sum_{n\geq 0}\frac{q^{n^2+n}}{(q; q)_n}=\frac{1}{(q^2, q^3; q^5)_{\infty}}.\label{id:RR}
\end{align}
The identities in \eqref{id:RR} were first discovered and proved by Rogers \cite{rog94}, subsequently rediscovered by Ramanujan \cite{ram14, ram19}, and independently established by Schur \cite{sch17}. Rogers-Ramanujan type identities play a fundamental role across combinatorics, number theory, the representation theory of affine Lie algebras, and solvable lattice models in statistical mechanics; for a comprehensive exposition, see Sills' monograph \cite{sil18}. A monumental collection of such identities was systematically developed through the Bailey machinery by Slater \cite{Sla52}.

From the perspective of partition theory, the product sides of \eqref{id:RR} generate partitions whose parts satisfy specific congruence conditions modulo $5$, whereas the series sides reflect difference conditions imposed on adjacent parts (namely, that adjacent parts differ by at least $2$). This profound duality between congruence and difference conditions was extended to arbitrary odd moduli by Gordon \cite{Gor61} and Andrews \cite{And74}, and to even moduli by Bressoud \cite{bre80}. In the study of higher-rank or higher-modulus partition theorems, multisums arise naturally: each independent summation index accounts for parts lying in designated residue classes or structural blocks, while the denominator factors keep track of admissible part increments.

Analytically, a central problem connecting Rogers-Ramanujan type series with the theory of modular forms is the systematic classification of modular $q$-hypergeometric series. This investigation is intimately tied to the study of Nahm sums and their deep interplay with conformal field theory, the Bloch group, and the dilogarithm function \cite{Nah07, Zag07}. However, this analytic perspective requires careful demarcation: Nahm's original conjecture is known to fail in general, as explicit counterexamples were constructed by Vlasenko and Zwegers \cite{VZ11}. Consequently, multivariable parameterized series cannot simply be presumed to yield modular forms without dedicated analysis. The objective of the present paper is combinatorial: we seek to explain how specific multisums arise from partition structures and to construct explicit bijective proofs for their product sides.

Recently, Wang and his collaborators \cite{CW23, CW24, LW24, WW23, WW25, wan24} carried out an extensive investigation into these structures, deriving a remarkable series of Rogers-Ramanujan type identities by means of algebraic and analytic machinery, including the constant term method, Bailey's lemma, recurrence relations, and complex contour integrals. A natural problem then arises: can one provide combinatorial interpretations of these identities in terms of integer partitions and construct explicit bijective proofs? Such an investigation follows the fruitful line of combinatorial refinements and difference-condition techniques developed in, for example, \cite{AG95, kur10, kur19, kur21, KS22, war97}. To formulate these interpretations, we recall the standard notions of partition theory; further details may be found in Andrews' treatise \cite{andtp}.

For a given non-negative integer $n$, a \emph{partition} $\lambda$ of $n$ is a weakly increasing sequence of positive integers summing to $n$. We write $\lambda = \lambda_1 + \lambda_2 + \cdots + \lambda_m$ or $\lambda = (\lambda_1, \lambda_2, \dots, \lambda_m)$ with $\lambda_1 \leq \lambda_2 \leq \cdots \leq \lambda_m$, where the \emph{weight} of $\lambda$ is denoted by $|\lambda| = n$ and each integer $\lambda_i$ is called a \emph{part} of $\lambda$ for $1 \leq i \leq m$. The number of parts $m$ is called the \emph{length} of $\lambda$, denoted by $\ell(\lambda)$. In this paper, labeled partitions may occur, so we let \(\ell_x(\la)\) denote the number of parts of \(\la\) labeled as \(x\). We let $\sP(n)$ denote the set of partitions of $n$, write $p(n) := |\sP(n)|$, and set $\sP := \bigcup_{n \geq 0} \sP(n)$. Let $\epsilon$ denote the unique empty partition of $0$. For any two partitions $\lambda, \mu \in \sP$, let $\lambda \oplus \mu$ denote the partition obtained by taking the multiset union of the parts of $\lambda$ and $\mu$ and arranging them in non-decreasing order.

This paper can be considered as a sequel to \cite{FL241, FL242, FL261, li25}, in which we combinatorially prove further entries from Wang's papers. We focus here on \emph{strict partitions}, which are partitions with distinct parts. We denote the set of strict partitions by $\sD$, and let $\sD(n)$ stand for the subset of strict partitions of weight $n$. Fu and the author gave a combinatorial proof of the following identity in \cite[Theorem 1.2]{FL241}:
\begin{align}\label{FL}
\sum_{i, j\geq 0}\frac{x^{i+2j}q^{i^2+2ij+2j^2+j}}{(q; q)_{i}(q^2; q^2)_{j}} &=(-xq; q)_{\infty}.
\end{align}
The neat combinatorial structure of \eqref{FL} motivates the study of analogous identities whose series sides involve triple sums.

\begin{theorem}[{cf.~\cite[Eq.~(6.7)]{CW24}}]\label{thm:1}
We have 
\begin{align}
\sum_{i, j, k\geq 0}\frac{x^{i+j+2k}q^{3i^2+3j^2+4k^2+2ij+4ik+4jk-2j}}{(q^4; q^4)_i(q^4; q^4)_j(q^4; q^4)_k}=(-xq; q^2)_{\infty}.\label{eq1}
\end{align}
\end{theorem}

\begin{theorem}[{cf.~\cite[Eq.~(4.5)]{LW24}}]\label{thm:2}
We have
\begin{align}
\sum_{i, j, k\geq 0}\frac{(-1)^ix^{k}y^{i+j}q^{i^2+2j^2+k^2+2ij+2ik+2jk+i}}{(q^2; q^2)_i(q^2; q^2)_j(q^2; q^2)_k}=(-xq; q^2)_{\infty}.\label{eq2}
\end{align}
\end{theorem}

\begin{theorem}[{cf.~\cite[Eqs.~(4.23), (4.24)]{LW24}}]\label{thm:3}
We have
\begin{align}
\sum_{i, j, k\geq 0}\frac{x^{j+k}q^{i^2/2+3j^2+2k^2+2ij+2ik+4jk+i/2}}{(q; q)_i(q^2; q^2)_j(q^2; q^2)_k}&=(-q; q)_{\infty}(-xq^2; q^4)_{\infty},\label{eq3-1}\\
\sum_{i, j, k\geq 0}\frac{x^{j+k}q^{i^2/2+3j^2+2k^2+2ij+2ik+4jk+3i/2+2j}}{(q; q)_i(q^2; q^2)_j(q^2; q^2)_k}&=(-q^2; q)_{\infty}(-xq^2; q^4)_{\infty}.\label{eq3-2}
\end{align}
\end{theorem}

\begin{theorem}[{cf.~\cite[Eqs.~(4.26), (4.30)]{LW24}}]\label{thm:4}
We have
\begin{align}
\sum_{i, j, k\geq 0}\frac{(-1)^jx^{i}y^{j+k}q^{i^2+6j^2+4k^2+4ij+4ik+8jk+2k}}{(q^2; q^2)_i(q^4; q^4)_j(q^4; q^4)_k}&=(-xq; q^2)_{\infty},\label{eq4-1}\\
\sum_{i, j, k\geq 0}\frac{(-1)^{i+j}x^{i+2j+2k}q^{i^2+6j^2+4k^2+4ij+4ik+8jk-i-4j-2k}}{(q^2; q^2)_i(q^4; q^4)_j(q^4; q^4)_k}&=(x; q^2)_{\infty}.\label{eq4-2}
\end{align}
\end{theorem}

\begin{theorem}[{cf.~\cite[Eqs.~(4.27), (4.29)]{LW24}}]\label{thm:5}
We have
\begin{align}
\sum_{i, j, k\geq 0}\frac{(-1)^{i+j}x^{i+2j+k}q^{i^2+6j^2+4k^2+4ij+4ik+8jk-2i-6j-5k}}{(q^2; q^2)_i(q^4; q^4)_j(q^4; q^4)_k}&=(xq; q^4)_{\infty},\label{eq5-1}\\
\sum_{i, j, k\geq 0}\frac{(-1)^{i+j}x^{i+2j+k}q^{i^2+6j^2+4k^2+4ij+4ik+8jk-2j-k}}{(q^2; q^2)_i(q^4; q^4)_j(q^4; q^4)_k}&=(xq; q^4)_{\infty}.\label{eq5-2}
\end{align}
\end{theorem}

\begin{theorem}[{cf.~\cite[Eqs.~(3.190), (3.192), (3.196)]{WW25}}]\label{thm:6}
We have
\begin{align}
\sum_{i, j, k\geq 0}\frac{q^{3i^2/2+2j^2+4k^2+2ij+4ik+4jk-3i/2-j-2k}}{(q; q)_i(q^2; q^2)_j(q^2; q^2)_k}&=(-1; q)_{\infty},\label{eq6-1}\\
\sum_{i, j, k\geq 0}\frac{q^{3i^2+4j^2+8k^2+4ij+8ik+8jk-2i-2k}}{(q^2; q^2)_i(q^4; q^4)_j(q^4; q^4)_k}&=(-q; q^2)_{\infty},\label{eq6-3}\\
\sum_{i, j, k\geq 0}\frac{q^{3i^2/2+2j^2+4k^2+2ij+4ik+4jk-i/2+j}}{(q; q)_i(q^2; q^2)_j(q^2; q^2)_k}&=\frac{1}{(q; q^2)_{\infty}}.\label{eq6-4}
\end{align}
\end{theorem}

\begin{theorem}[{cf.~\cite[Eqs.~(4.26)--(4.31)]{wan24}}]\label{thm:7}
We have
\begin{align}
\sum_{i,j,k \ge 0} \frac{q^{3i^2/2 + j^2 + k^2/2 + 2ij + ik + jk}}{(q;q)_i (q;q)_j (q;q)_k} &= \frac{(-q^{\frac{1}{2}};q)_\infty}{(q,q^4;q^5)_\infty},\label{eq7-1}\\
\sum_{i,j,k \ge 0} \frac{q^{3i^2/2 + j^2 + k^2/2 + 2ij + ik + jk + i/2 + k/2}}{(q;q)_i (q;q)_j (q;q)_k} &= \frac{(-q;q)_\infty}{(q,q^4;q^5)_\infty},\label{eq7-2}\\
\sum_{i,j,k \ge 0} \frac{q^{3i^2/2 + j^2 + k^2/2 + 2ij + ik + jk - i/2 - k/2}}{(q;q)_i (q;q)_j (q;q)_k} &= 2\frac{(-q;q)_\infty}{(q,q^4;q^5)_\infty}, \label{eq7-3}\\
\sum_{i,j,k \ge 0} \frac{q^{3i^2/2 + j^2 + k^2/2 + 2ij + ik + jk + i + j}}{(q;q)_i (q;q)_j (q;q)_k} &= \frac{(-q^{\frac{1}{2}};q)_\infty}{(q^2,q^3;q^5)_\infty},\label{eq7-4}\\
\sum_{i,j,k \ge 0} \frac{q^{3i^2/2 + j^2 + k^2/2 + 2ij + ik + jk + 3i/2 + j + k/2}}{(q;q)_i (q;q)_j (q;q)_k} &= \frac{(-q;q)_\infty}{(q^2,q^3;q^5)_\infty},\label{eq7-5} \\
\sum_{i,j,k \ge 0} \frac{q^{3i^2/2 + j^2 + k^2/2 + 2ij + ik + jk + i/2 + j - k/2}}{(q;q)_i (q;q)_j (q;q)_k} &= 2\frac{(-q;q)_\infty}{(q^2,q^3;q^5)_\infty}. \label{eq7-6}
\end{align}
\end{theorem}

We observe that all the identities in Theorems \ref{thm:6} and \ref{thm:7} involve solely the modular parameter $q$ and contain no auxiliary variables. In this paper, we establish unified parameterized generalisations of these families. For the identities in Theorem \ref{thm:7}, we develop the following two-parameter identity.
\begin{theorem}\label{thm:9}
We have
\begin{align}
\sum_{i, j, k\geq 0}\frac{x^{i+k}y^{i+j}q^{3i^2/2+j^2+k^2/2+2ij+ik+jk+i/2+k/2}}{(q; q)_i(q; q)_j(q; q)_k}=(-xq; q)_{\infty}\sum_{n\geq 0}\frac{y^nq^{n^2}}{(q; q)_n}.\label{eq9}
\end{align}
\end{theorem}
All six identities in Theorem \ref{thm:7} can be recovered from \eqref{eq9} via the following specialisations, employing the two classical Rogers-Ramanujan identities \eqref{id:RR}:
\begin{align*}
&\eqref{eq9}\xrightarrow[\text{by the first identity in \eqref{id:RR}}]{\text{$y\mapsto 1$ and $x\mapsto q^{-\frac{1}{2}}$}}\eqref{eq7-1},\quad \eqref{eq9}\xrightarrow[\text{by the second identity in \eqref{id:RR}}]{\text{$y\mapsto q$ and $x\mapsto q^{-\frac{1}{2}}$}}\eqref{eq7-4},\\
&\eqref{eq9}\xrightarrow[\text{by the first identity in \eqref{id:RR}}]{\text{$y\mapsto 1$ and $x\mapsto 1$}}\eqref{eq7-2},\quad \eqref{eq9}\xrightarrow[\text{by the second identity in \eqref{id:RR}}]{\text{$y\mapsto q$ and $x\mapsto 1$}}\eqref{eq7-5},\\
&\eqref{eq9}\xrightarrow[\text{by the first identity in \eqref{id:RR}}]{\text{$y\mapsto 1$ and $x\mapsto q^{-1}$}}\eqref{eq7-3},\quad \eqref{eq9}\xrightarrow[\text{by the second identity in \eqref{id:RR}}]{\text{$y\mapsto q$ and $x\mapsto q^{-1}$}}\eqref{eq7-6}.
\end{align*}

For the family in Theorem \ref{thm:6}, we construct a multisum identity equipped with a refinement parameter and an arbitrary number of summation indices.

\begin{theorem}\label{thm:8}
For any integer $m\geq 0$, we have
\begin{align}
\sum_{i, k_1, k_2, \dots, k_m\geq 0}\frac{x^{i+\sum_{t=1}^m 2k_t}q^{M(i, k_1, \dots, k_m)}}{(q; q)_i(q^2; q^2)_{k_1}(q^2; q^2)_{k_2}\cdots (q^2; q^2)_{k_m}}=(-xq; q)_{\infty},\label{eq8}
\end{align}
where
\begin{align}\label{M}
M(i, k_1, \dots, k_m)&:=\sum_{t=1}^m 2tk_t^2+\frac{m+1}{2}i^2+\sum_{t=1}^m\sum_{s=1}^{t-1}4sk_sk_t\nonumber\\
&\qquad\qquad\qquad +\sum_{t=1}^m(2-t)k_t+\left(\frac{1-m}{2}+\sum_{t=1}^m 2tk_t\right)i.
\end{align}
\end{theorem}
In particular, when $m=0$, \eqref{eq8} reduces to Euler's classical $q$-exponential identity:
\begin{align}
\sum_{i\geq 0}\frac{x^iq^{i^2/2+i/2}}{(q; q)_i}=(-xq; q)_{\infty}.
\end{align}
When $m=1$ and letting $k_1\mapsto j$, \eqref{eq8} becomes \eqref{FL}, which appears in an equivalent form in Cao and Wang \cite[Eq.~(3.29)]{CW23}. When $m=2$, relabelling the indices $k_1\mapsto j$ and $k_2\mapsto k$, identity \eqref{eq8} reduces to
\begin{align}\label{eq8-1}
\sum_{i, j, k\geq 0}\frac{x^{i+2j+2k}q^{3i^2/2+2j^2+4k^2+2ij+4ik+4jk-i/2+j}}{(q; q)_i(q^2; q^2)_j(q^2; q^2)_k}=(-xq; q)_{\infty}.
\end{align}
All the identities in Theorem \ref{thm:6} can be recovered from \eqref{eq8-1} via the following specialisations:
\begin{align*}
&\eqref{eq8-1}\xrightarrow{\text{$x\mapsto q^{-1}$}}\eqref{eq6-1},\quad \eqref{eq8-1}\xrightarrow{\text{$q\mapsto q^2$ and $x\mapsto q^{-1}$}}\eqref{eq6-3},\\
&\eqref{eq8-1}\xrightarrow[\text{by Euler's identity $1/(q; q^2)_{\infty}=(-q; q)_{\infty}$}]{\text{$x\mapsto 1$}}\eqref{eq6-4}.
\end{align*}

The remainder of this paper is organised as follows. Sections \ref{sec:1}, \ref{sec:2}, \ref{sec:3}, \ref{sec:4}, \ref{sec:5}, \ref{sec:6}, and \ref{sec:7} are devoted to establishing combinatorial interpretations and constructing explicit bijective proofs for the identities in Theorems \ref{thm:1}, \ref{thm:2}, \ref{thm:3}, \ref{thm:4}, \ref{thm:5}, \ref{thm:9}, and \ref{thm:8}, respectively. In addition, in Section \ref{sec:7} we establish a combinatorial decomposition of strict partitions with multiple statistics. Finally, we conclude the paper with brief summary remarks in Section \ref{sec:conclusion}.

\section{Combinatorial proof of Theorem \ref{thm:1}}\label{sec:1}

We first perform an elementary algebraic substitution in \eqref{eq1}: replacing $x$ by $xq$ and subsequently replacing $q^2$ by $q$ yields
\begin{align}
\sum_{i, j, k\geq 0}\frac{x^{i+j+2k}q^{3i^2/2+3j^2/2+2k^2+ij+2ik+2jk+i/2-j/2+k}}{(q^2; q^2)_i(q^2; q^2)_j(q^2; q^2)_k}=(-xq; q)_{\infty}.\label{sec1-1}
\end{align}
The infinite product on the right-hand side of \eqref{sec1-1} is the generating function for strict partitions, with the parameter $x$ recording the number of parts; that is,
\begin{align}\label{sec1-strict}
\sum_{\la\in \sD}x^{\ell(\la)}q^{|\la|}=(-xq; q)_{\infty}.
\end{align}

We now turn our attention to the series on the left-hand side of \eqref{sec1-1}, which factors naturally as
\begin{align}\label{sec1-3}
\text{LHS of }\eqref{sec1-1}=\sum_{k\geq 0}\frac{x^{2k}q^{2k^2+k}}{(q^2; q^2)_k}\sum_{j\geq 0}\frac{x^{j}q^{(2k+1)+\cdots +(2k+3j-2)}}{(q^2; q^2)_j}\sum_{i\geq 0}\frac{x^iq^{(2k+j+2)+\cdots +(2k+j+3i-1)}}{(q^2; q^2)_i}.
\end{align}
Here and throughout, a sum of base parts is understood to be empty whenever the corresponding upper index is zero. For fixed $k, j\geq 0$, the factor $q^{(2k+1)+\cdots +(2k+3j-2)}$ prescribes base parts of the form $2k+1, \dots, 2k+3j-2$, while $1/(q^2; q^2)_j$ generates partitions into even parts with at most $j$ parts. Padding with zeros where necessary, these even parts can be uniquely written as $0\leq 2t_1'\leq \cdots\leq 2t_j'$. Adding the $m$-th even part $2t_m'$ to the $m$-th base part, the factor $q^{(2k+1)+\cdots +(2k+3j-2)}/(q^2; q^2)_j$ serves as the generating function for the partition family
\begin{align*}
\sA^{\I}_{j, k}:=\{\la\in \sD: \ell(\la)=j,\ \la_m=2k+2t_m'+3m-2\ (1\leq m\leq j),\ 0\leq t_1'\leq \cdots\leq t_j'\}.
\end{align*}
By convention, $\sA_{0, k}^{\I}=\{\ep\}$. To facilitate the bijective construction, we introduce the strictly increasing indices $t_m:=t_m'+m-1$ for $1\leq m\leq j$. This yields the equivalent characterization
\begin{align*}
\sA^{\I}_{j, k}=\{\la\in \sD: \ell(\la)=j,\ \la_m=2k+2t_m+m\ (1\leq m\leq j),\ 0\leq t_1< \cdots <t_j\}.
\end{align*}
Similarly, for fixed $k, j, i\geq 0$, the $i$-th sum in \eqref{sec1-3} generates the family
\begin{align*}
\sB_{i, j, k}^{\I}:=\{\mu\in \sD: \ell(\mu)=i,\ \mu_m=2k+(2s_m+1)+j+m\ (1\leq m\leq i),\ 0\leq s_1<\cdots<s_i\},
\end{align*}
with $\sB_{0, j, k}^{\I}=\{\ep\}$. In both families, the variable $x$ marks the number of parts.

For any subset $A\subset \bN_0$, let $\o(A)$ denote the number of odd elements in $A$. We define
\begin{align*}
\sR_{n, k}&:=\left\{\ga\in \sD: \ell(\ga)=n,\ \ga_m=2k+r_{m}+m\ (1\leq m\leq n),\ 0\leq r_1<\cdots <r_{n} \right\},\\
\sR_{i, j, k}^{\I}&:=\left\{\ga\in \sR_{i+j, k}: \o(\{r_1, \dots, r_{i+j}\})=i \right\},
\end{align*}
setting $\sR_{0, k}=\sR_{0, 0, k}^{\I}=\{\ep\}$.

To establish the identity combinatorially, we aim to construct a bijection
\begin{align*}
\varphi^{\I'}: \sA^{\I}_{j, k}\times \sB_{i, j, k}^{\I}&\longrightarrow \sR_{i, j, k}^{\I},\\
(\la, \mu)&\longmapsto \ga.
\end{align*}
To make the bijection fully transparent and free of extraneous parameters, we remove the common shift $2k$ from every part of each partition in $\sA^{\I}_{j, k}$, $\sB_{i, j, k}^{\I}$, and $\sR_{i, j, k}^{\I}$. This subtraction yields the families $\sA^{\I}_{j}$, $\sB_{i, j}^{\I}$, and $\sR_{i, j}^{\I}$, inducing canonical bijections denoted respectively by $f_{\sA}$, $f_{\sB}$, and $f_{\sR}$. These maps assemble into the commutative diagram displayed in Figure \ref{fig:bijection-varphi}.

\begin{figure}[htbp]
\centering
\begin{tikzpicture}[
  scale=1.2,
  every node/.style={font=\large},
  set/.style={draw, rounded corners, minimum width=2.2cm, minimum height=0.9cm, align=center},
  arr/.style={<->, >={Stealth[length=2.5mm]}, thick}
]

\node[set] (A) at (0,3) {$\sA^{\I}_{j, k}\times \sB_{i, j, k}^{\I}$};
\node[set] (B) at (5,3) {$\sR_{i, j, k}^{\I}$};
\node[set] (C) at (0,0) {$\sA^{\I}_{j}\times \sB_{i, j}^{\I}$};
\node[set] (D) at (5,0) {$\sR_{i, j}^{\I}$};

\draw[arr] (A) -- node[above]{$\varphi^{\I'}$} (B);
\draw[arr] (C) -- node[below]{$\varphi^{\I}$} (D);
\draw[arr] (A) -- node[left]{$f_{\sA}\times f_{\sB}$} (C);
\draw[arr] (B) -- node[right]{$f_{\sR}$} (D);

\end{tikzpicture}
\caption{Commutative diagram relating the bijections $\varphi^{\I'}$ and $\varphi^{\I}$.}
\label{fig:bijection-varphi}
\end{figure}
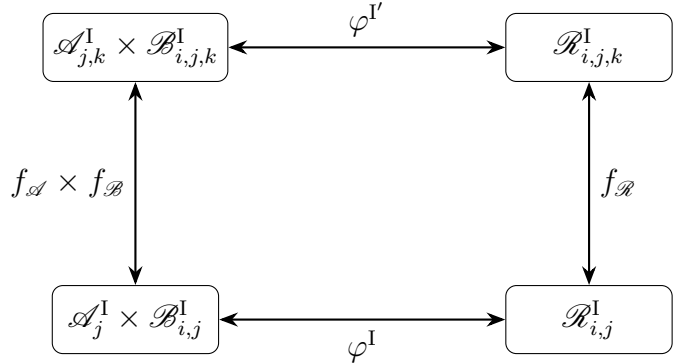

Consequently, establishing the bijection $\varphi^{\I'}$ reduces to proving the following lemma.

\begin{lemma}\label{lem:varphi-I}
For any $i, j\geq 0$, there exists a bijection
\begin{align*}
\varphi^{\I}: \sA^{\I}_{j}\times \sB_{i, j}^{\I}&\longrightarrow \sR_{i, j}^{\I},\\
(\la, \mu)&\longmapsto \ga,
\end{align*}
such that $|\la|+|\mu|=|\ga|$ and $\ell(\la)+\ell(\mu)=\ell(\ga)$.
\end{lemma}

With Lemma \ref{lem:varphi-I} in hand, we obtain the generating function identity
\begin{align}\label{sec1-2}
\sum_{(\la, \mu)\in \bigcup_{i, j\geq 0}(\sA^{\I}_{j, k}\times \sB_{i, j, k}^{\I})}x^{\ell(\la)+\ell(\mu)}q^{|\la|+|\mu|}&=\sum_{\ga\in \bigcup_{i, j\geq 0}\sR_{i, j, k}^{\I}}x^{\ell(\ga)}q^{|\ga|}=\sum_{n\geq 0}x^n\sum_{\ga\in \bigcup_{i+j=n}\sR_{i, j, k}^{\I}}q^{|\ga|}\nonumber\\
&=\sum_{n\geq 0}x^n\sum_{\ga\in \sR_{n, k}}q^{|\ga|}=\sum_{n\geq 0}\frac{x^{n}q^{n^2+2kn}}{(q; q)_n}.
\end{align}
Indeed, the union over all pairs $(i,j)$ with $i+j=n$ partitions $\sR_{n, k}$ according to the number of odd elements in the underlying set $\{r_1, \dots, r_n\}$. Writing $r_m=(m-1)+u_m$ with $0\leq u_1\leq \cdots\leq u_n$ accounts for the standard denominator factor $(q; q)_n$. Substituting \eqref{sec1-2} into \eqref{sec1-3} reduces identity \eqref{sec1-1} to \eqref{FL}, for which a bijective proof was established by Fu and the author in \cite{FL241}. This completes the combinatorial proof of Theorem \ref{thm:1}.

\begin{proof}[Proof of Lemma \ref{lem:varphi-I}]
We first define the forward mapping $\varphi^{\I}$. Given a pair of partitions $(\la, \mu)\in \sA^{\I}_{j}\times \sB_{i, j}^{\I}$, we carry out the following two steps:
\begin{enumerate}[(1)]
\item Extract the $j$ even integers $0\leq 2t_1<\cdots <2t_j$ from $\la$ to form the set $E:=\{2t_m\}_{m=1}^j$.
\item Extract the $i$ odd integers $0< 2s_1+1<\cdots <2s_i+1$ from $\mu$ to form the set $O:=\{2s_m+1\}_{m=1}^i$.
\end{enumerate}
Because $E\cap O=\emptyset$, their union $E\cup O$ comprises $i+j$ distinct nonnegative integers. Arranging these elements in strictly increasing order, we obtain
\[
0\leq r_1<r_2<\cdots <r_{i+j}.
\]
We then set the parts of the partition $\ga$ to be $\ga_m:=r_m+m$ for $1\leq m\leq i+j$.

By construction, exactly $i$ elements in $\{r_1, \dots, r_{i+j}\}$ originate from $O$ and are therefore odd; thus $\ga\in \sR_{i, j}^{\I}$. It is immediate that $\ell(\la)+\ell(\mu)=j+i=\ell(\ga)$. Furthermore, the operation preserves total weight:
\begin{align*}
|\la|+|\mu|&=\sum_{m=1}^j(2t_m+m)+\sum_{m=1}^i(2s_m+1+j+m)\\
&=\sum_{m=1}^j 2t_m+\sum_{m=1}^i(2s_m+1)+\left(\sum_{m=1}^j m+ij+\sum_{m=1}^i m\right)\\
&=\sum_{m=1}^{i+j}r_m+\frac{(i+j)(i+j+1)}{2}=\sum_{m=1}^{i+j}(r_m+m)=|\ga|.
\end{align*}

Conversely, we construct the inverse mapping $(\varphi^{\I})^{-1}$. Given $\ga\in \sR_{i, j}^{\I}$, we first determine the set $R:=\{\ga_m-m : 1\leq m\leq i+j\}=\{r_1, \dots, r_{i+j}\}$. Since $\o(R)=i$, the set $R$ contains precisely $j$ even integers and $i$ odd integers. We then perform the inverse operations:
\begin{enumerate}[(1)]
\item Collect the $j$ even integers in $R$ and sort them in strictly increasing order as $0\leq r_{m_1}<\cdots<r_{m_j}$. Setting $t_a:=r_{m_a}/2$ for $1\leq a\leq j$ produces a sequence satisfying $0\leq t_1<\cdots <t_j$. We reconstruct $\la$ by setting $\la_a:=2t_a+a$ for $1\leq a\leq j$.
\item Collect the $i$ odd integers in $R$ and sort them in strictly increasing order as $0< r_{n_1}<\cdots<r_{n_i}$. Setting $s_b:=(r_{n_b}-1)/2$ for $1\leq b\leq i$ produces a sequence satisfying $0\leq s_1<\cdots <s_i$. We reconstruct $\mu$ by setting $\mu_b:=(2s_b+1)+j+b$ for $1\leq b\leq i$.
\end{enumerate}
It is straightforward to verify that $(\la, \mu)\in \sA_{j}^{\I}\times \sB_{i, j}^{\I}$, with $|\la|+|\mu|=|\ga|$ and $\ell(\la)+\ell(\mu)=\ell(\ga)$.

Finally, we show that the two constructions are two-sided inverses of each other:
\begin{itemize}
\item $(\varphi^{\I})^{-1}\circ \varphi^{\I}=\mathrm{id}_{\sA_{j}^{\I}\times \sB_{i, j}^{\I}}$. For any $(\la, \mu)\in \sA_{j}^{\I}\times \sB_{i, j}^{\I}$, the mapping $\varphi^{\I}$ merges the even set $E$ and the odd set $O$ into $R$. Because the even and odd integers form disjoint residue classes modulo $2$, sieving $R$ by parity uniquely recovers $E$ and $O$. Since the ordering of elements in $E$ and $O$ uniquely determines the parameters $t_m$ and $s_n$, the inverse map reconstructs the original pair $(\la, \mu)$ without loss.

\item $\varphi^{\I}\circ (\varphi^{\I})^{-1}=\mathrm{id}_{\sR_{i, j}^{\I}}$. For any $\ga\in \sR_{i, j}^{\I}$, the underlying set $R$ contains exactly $i$ odd integers and $j$ even integers, so $(\varphi^{\I})^{-1}$ uniquely splits $R$ into $E$ and $O$. Merging and sorting $E\cup O$ recovers the original ordered set $R$, thereby reconstructing $\ga$ identically.
\end{itemize}
When $i=0$ or $j=0$, the construction remains valid by setting the corresponding list of parts to be empty ($\ep$). This completes the proof of Lemma \ref{lem:varphi-I}.
\end{proof}

\begin{example}
Table \ref{tab:bijection_w14_concise} displays the complete one-to-one correspondence under the bijection $\varphi^{\I}$ at total weight $|\la|+|\mu|=|\ga|=14$.

\begin{table}[htbp]
\centering
\renewcommand{\arraystretch}{1.25}
\begin{tabular}{ccc|ccc}
\toprule
 $(\lambda, \mu)$ & & $\gamma $ & $(\lambda, \mu) $ & & $\gamma $\\
\midrule
 $(\epsilon, 14)$         & $\longleftrightarrow$ & $14$      & $(7, 7)$                & $\longleftrightarrow$ & $6 + 8$  \\
 $(1, 13)$               & $\longleftrightarrow$ & $1 + 13$  & $(1, 3 + 10)$           & $\longleftrightarrow$ & $1 + 3 + 10$ \\
 $(11, 3)$               & $\longleftrightarrow$ & $2 + 12$  & $(1 + 4 + 9, \epsilon)$ & $\longleftrightarrow$ & $1 + 4 + 9$  \\
 $(3, 11)$               & $\longleftrightarrow$ & $3 + 11$  & $(1, 5 + 8)$            & $\longleftrightarrow$ & $1 + 5 + 8$  \\
 $(9, 5)$                & $\longleftrightarrow$ & $4 + 10$  & $(3, 3 + 8)$            & $\longleftrightarrow$ & $2 + 4 + 8$  \\
 $(5, 9)$                & $\longleftrightarrow$ & $5 + 9$   & $(5, 3 + 6)$            & $\longleftrightarrow$ & $2 + 5 + 7$  \\
\bottomrule
\end{tabular}
\vspace{0.5em}
\caption{The bijection $\varphi^{\mathrm{I}}$ for total weight $14$.}
\label{tab:bijection_w14_concise}
\end{table}
\end{example}

Finally, combining Lemma \ref{lem:varphi-I} with the two identities in \eqref{id:RR} yields the following pair of double-sum Rogers-Ramanujan type identities.

\begin{corollary}
We have
\begin{align}
\sum_{i, j\geq 0}\frac{q^{3i^2/2+3j^2/2+ij+i/2-j/2}}{(q^2; q^2)_i(q^2; q^2)_j}&=\frac{1}{(q, q^4; q^5)_{\infty}},\\
\sum_{i, j\geq 0}\frac{q^{3i^2/2+3j^2/2+ij+3i/2+j/2}}{(q^2; q^2)_i(q^2; q^2)_j}&=\frac{1}{(q^2, q^3; q^5)_{\infty}}.
\end{align}
\end{corollary}

\section{Combinatorial proof of Theorem \ref{thm:2}}\label{sec:2}

We first transform \eqref{eq2} by replacing $x$ with $xq$ and then replacing $q^2$ with $q$. This yields
\begin{align}
\sum_{i, j, k\geq 0}\frac{(-1)^i x^k y^{i+j} q^{i^2/2+j^2+k^2/2+ij+ik+jk+i/2+k/2}}{(q; q)_i(q; q)_j(q; q)_k} = (-xq; q)_{\infty}.\label{sec2-1}
\end{align}
The right-hand side of \eqref{sec2-1} carries the same combinatorial interpretation as in \eqref{sec1-strict}. We decompose the left-hand side of \eqref{sec2-1} as
\begin{align}\label{sec2-2}
\text{LHS of }\eqref{sec2-1} = \sum_{j, k\geq 0}\frac{x^k y^j q^{j^2+k^2/2+jk+k/2}}{(q; q)_j(q; q)_k}\sum_{i\geq 0}\frac{(-y)^i q^{(j+k+1)+\cdots +(j+k+i)}}{(q; q)_i}.
\end{align}
For fixed $j, k\geq 0$, let $\sB_{j, k}^{\II}$ denote the set of strict partitions $\mu$ whose parts are all labeled $y$ and whose smallest part, when non-empty, is at least $j+k+1$. In particular, the empty partition satisfies $\ep \in \sB_{j, k}^{\II}$. By the standard generating function for strict partitions,
\begin{align}\label{sec2-4}
\sum_{\mu\in \sB_{j, k}^{\II}}(-y)^{\ell(\mu)}q^{|\mu|} = \sum_{i\geq 0}\frac{(-y)^i q^{(j+k+1)+\cdots +(j+k+i)}}{(q; q)_i}.
\end{align}

Next, we interpret the outer sum of \eqref{sec2-2}. Note that subscripts $x$ and $y$ denote labels and do not affect the numerical values or order of parts. We define the base partition
\begin{align}\label{beta-II}
\be^{\II, (j, k)} = 1_x + 2_x + \cdots + k_x + (k+1)_y + (k+3)_y + \cdots + (k+2j-1)_y,
\end{align}
with the conventional interpretation when $j=0$ or $k=0$. The weight of $\be^{\II, (j, k)}$ is
\begin{align*}
|\be^{\II, (j, k)}| = j^2 + \frac{k^2}{2} + jk + \frac{k}{2}.
\end{align*}
Let $\sP_n = \{\la\in \sP : \ell(\la)\leq n\}$. It follows that
\begin{align}\label{sec2-3}
\sum_{(\be^{\II, (j, k)}, \ga^1, \ga^2)\in \{\be^{\II, (j, k)}\}\times \sP_j\times \sP_k} x^k y^j q^{|\be^{\II, (j, k)}|+|\ga^1|+|\ga^2|} = \frac{x^k y^j q^{j^2+k^2/2+jk+k/2}}{(q; q)_j(q; q)_k}.
\end{align}

Let $\sA_{j, k}^{\II}$ denote the set of strict partitions $\la = \la_1 + \cdots + \la_{j+k}$ with exactly $j$ parts labeled $y$ and $k$ parts labeled $x$, subject to the condition that
\begin{align*}
\la_{m+1} - \la_m \geq 2 \quad \text{whenever $\la_m$ is labeled $y$ and $m < j+k$.}
\end{align*}
One can easily check that $\beta^{\II, (j, k)}$ is the partition with the smallest weight in $\sA^{\II}_{j, k}$. We write $\ell_y(\la)$ for the number of $y$-labeled parts in $\la$.

\begin{lemma}\label{lem:sec2-psiII}
For $j, k\geq 0$, there exists a bijection
\begin{align*}
\psi^{\II}\colon \{\be^{\II, (j, k)}\}\times \sP_j\times \sP_k &\longrightarrow \sA_{j, k}^{\II},\\
(\be^{\II, (j, k)}, \ga^{1}, \ga^{2}) &\longmapsto \la,
\end{align*}
such that
$|\be^{\II, (j, k)}|+|\ga^{1}|+|\ga^{2}|=|\la|$,
$\ell(\be^{\II, (j, k)})=\ell(\la)$, and
$\ell_y(\be^{\II, (j, k)})=\ell_y(\la)$.
\end{lemma}

\begin{proof}
We represent all partitions with parts in weakly increasing order and pad the increment partitions with zeros when necessary:
\[
\ga^1 = (b_1, \ldots, b_j), \qquad \ga^2 = (a_1, \ldots, a_k),
\]
where $0 \leq b_1 \leq \cdots \leq b_j$ and $0 \leq a_1 \leq \cdots \leq a_k$. Empty lists are understood when $j=0$ or $k=0$.

First, we describe the forward construction of $\psi^{\II}$. Recall the base partition
\[
\be^{\II, (j, k)} = 1_x + 2_x + \cdots + k_x + (k+1)_y + (k+3)_y + \cdots + (k+2j-1)_y.
\]
We first add $b_s$ to the $s$-th $y$-part of $\be^{\II, (j, k)}$ for each $1 \leq s \leq j$. The resulting $y$-parts are
\[
(k + 2s - 1 + b_s)_y, \qquad 1 \leq s \leq j.
\]
Since $b_{s+1} - b_s \geq 0$, the gap between consecutive $y$-parts satisfies
\[
(k + 2s + 1 + b_{s+1}) - (k + 2s - 1 + b_s) = 2 + (b_{s+1} - b_s) \geq 2,
\]
so the difference conditions among the $y$-parts are preserved. Next, we process the $x$-parts sequentially in decreasing order $r = k, k-1, \ldots, 1$. For the $r$-th $x$-part (initially $r_x$), we perform $a_r$ forward moves using the following local operations:
\begin{description}
\item[Forward move] For an active part $t_x$, an ordinary forward move replaces $t_x$ by $(t+1)_x$, increasing the total weight by $1$.

\item[Adjustment] If the part immediately succeeding $t_x$ is $(t+1)_y$, an ordinary forward move would temporarily produce two parts of value $t+1$. We resolve this by the joint transition
\[
t_x, \ (t+1)_y \longmapsto t_y, \ (t+2)_x.
\]
This complete step increases the total weight by $1$ and moves the active $x$-part past the adjacent $y$-part.
\end{description}
We verify that these forward moves are globally well-defined and never collide with previously placed parts. Suppose parts $(r+1)_x, \ldots, k_x$ have already completed their forward moves, with $(r+1)_x$ now resting at its final position $z$. All $j$ $y$-parts initially resided strictly to the right of $k_x$, and any larger $x$-part ($> r+1$) only shifted $y$-parts to positions strictly greater than $r+1$. Thus, the only $y$-parts currently located between $r$ and $z$ are the $h$ $y$-parts that were crossed by $(r+1)_x$ during its $a_{r+1}$ forward moves. Since each adjustment advances an $x$-part by $2$ while an ordinary step advances it by $1$, the position of $(r+1)_x$ satisfies
\[
z = (r+1) + a_{r+1} + h.
\]
When moving $r_x$ forward, for it to reach or collide with $z$, it must cross all these $h$ intermediate $y$-parts. After $N$ forward steps crossing these $h$ $y$-parts, the position of $r_x$ becomes $r + (N - h) + 2h = r + N + h$. For a collision to occur, we would require $r + N + h \geq z = (r+1) + a_{r+1} + h$, which simplifies to $N \geq a_{r+1} + 1$. Because $a_r \leq a_{r+1} < a_{r+1} + 1$, the active part $r_x$ never reaches $(r+1)_x$, completely precluding collisions between $x$-parts. Local difference conditions with adjacent $y$-parts are verified straightforwardly: before an adjustment, any preceding part has value at most $t-1$ if labeled $x$ and at most $t-2$ if labeled $y$; after the adjustment, the new $y$-part $t_y$ satisfies the required gap from its left neighbor, the gap between $t_y$ and $(t+2)_x$ is $2$, and any subsequent part was originally at least $t+3$ and hence differs from $(t+2)_x$ by at least $1$. Thus, the forward algorithm yields a unique, admissible partition $\la \in \sA_{j, k}^{\II}$.

Conversely, we define the inverse map $(\psi^{\II})^{-1}$. Let $\la \in \sA_{j, k}^{\II}$, and denote its $x$-parts in strictly increasing order by $X_1 < X_2 < \cdots < X_k$. We restore these parts to their minimal positions in the increasing order $r = 1, \ldots, k$. For the $r$-th $x$-part, we repeatedly move it backwards until its value reaches $r$. The operations are:
\begin{description}
\item[Backward move] For an active part $t_x$, an ordinary backward move replaces $t_x$ by $(t-1)_x$, decreasing the total weight by $1$.

\item[Normalization] If the part immediately preceding $t_x$ is $(t-2)_y$, an ordinary backward move would leave an inadmissible gap of $1$ after this $y$-part. Instead, we perform the joint step
\[
(t-2)_y, \ t_x \longmapsto (t-2)_x, \ (t-1)_y.
\]
This is the exact inverse of an adjustment, decreasing the total weight by $1$ and swapping the $x$-part past the $y$-part.
\end{description}
At stage $r$, parts $1_x, \ldots, (r-1)_x$ have already been restored. The active part satisfies $t \geq r$. For an ordinary backward move, $t > r$ implies $t - 1 \geq r$. For a normalization, the preceding part is $(t-2)_y$. If $r = 1$, since all parts are positive integers, $t - 2 \geq 1 = r$. If $r > 1$, the preceding $y$-part must be strictly larger than the restored part $(r-1)_x$, giving $t - 2 \geq r$. Hence, normalization transforms $t_x$ into $(t-2)_x \geq r_x$ and never overshoots the target value $r$. The process terminates deterministically when the active part reaches $r_x$.

Let $a_r$ be the total number of backward steps (counting each normalization as one step) required to restore $X_r$ to $r_x$. We now prove that $a_1 \leq \cdots \leq a_k$. Let $C_r$ be the number of $y$-parts in $\la$ strictly smaller than $X_r$. The restoration of earlier $x$-parts ($s < r$) does not move any $y$-part past $X_r$, as each affected $y$-part increases by $1$ and remains bounded above by $s < X_r$. Once $X_r$ has been restored to $r_x$, all $y$-parts must lie strictly above $r_x$, because parts $1_x, \ldots, r_x$ form a sequence of consecutive integers with gap $1$, leaving no admissible site for a $y$-part (which requires gap at least $2$ to its successor). Therefore, during its restoration, $X_r$ must cross exactly $C_r$ $y$-parts via normalizations. Since each ordinary backward move decreases the value of the $x$-part by $1$ and each normalization decreases it by $2$, we obtain the algebraic invariant
\[
X_r - r = (a_r - C_r) + 2C_r = a_r + C_r \implies X_r = r + a_r + C_r.
\]
For $1 \leq r < k$, let $h = C_{r+1} - C_r \geq 0$ be the number of $y$-parts strictly between $X_r$ and $X_{r+1}$ in $\la$. By the difference conditions of $\sA_{j, k}^{\II}$, the gap between consecutive $y$-parts is at least $2$, and the gap from any $y$-part to its successor is at least $2$. Summing these gaps gives
\[
X_{r+1} - X_r \geq 1 + 2h.
\]
Substituting the algebraic invariant into this inequality, we have
\[
X_{r+1} - X_r = 1 + (a_{r+1} - a_r) + h \geq 1 + 2h \implies a_{r+1} - a_r \geq h \geq 0.
\]
Thus, $\ga^2 = (a_1, \ldots, a_k)$ satisfies $0 \leq a_1 \leq \cdots \leq a_k$, so $\ga^2 \in \sP_k$.

After all $k$ $x$-parts have been restored to $1_x, \ldots, k_x$, the remaining $y$-parts, denoted $Y_1 < \cdots < Y_j$, all lie strictly above $k_x$ and retain mutual gaps of at least $2$. Consequently,
\[
Y_s \geq k + 2s - 1, \qquad 1 \leq s \leq j.
\]
Setting $b_s = Y_s - (k + 2s - 1)$, the condition $Y_{s+1} - Y_s \geq 2$ immediately guarantees $b_{s+1} - b_s \geq 0$, so $\ga^1 = (b_1, \ldots, b_j) \in \sP_j$. Subtracting these increments recovers $\be^{\II, (j, k)}$.

Since the local transitions of forward moves and adjustments are the precise step-by-step inverses of backward moves and normalizations under unambiguous local neighbor conditions, and both algorithms are globally well-defined on their respective domains, $\psi^{\II}$ and $(\psi^{\II})^{-1}$ are two-sided inverses. Each forward step adds $1$ to the weight, and adding the increments $b_s$ adds $|\ga^1|$, while preserving the part labels and counts. We conclude that
\[
|\la| = |\be^{\II, (j, k)}| + |\ga^1| + |\ga^2|, \qquad \ell(\la) = j+k, \qquad \ell_y(\la) = j,
\]
which establishes the desired bijection.
\end{proof}

\begin{example}
Let
\begin{align*}
\be^{\II, (3, 2)} = 1_x + 2_x + 3_y + 5_y + 7_y, \qquad \ga^1 = 2 + 3 + 5, \qquad \ga^2 = 3 + 5.
\end{align*}
Adding $\ga^1$ yields $1_x + 2_x + 5_y + 8_y + 12_y$.
We then apply $\ga^2_2 = 5$ forward moves to $2_x$, followed by $\ga^2_1 = 3$ forward moves to $1_x$:
\begin{align*}
1_x + 2_x + 5_y + 8_y + 12_y 
&\xrightarrow{2\text{ moves of } 2_x} 1_x + \mathbf{4}_x + 5_y + 8_y + 12_y\\
&\xrightarrow{\text{adjustment}} 1_x + 4_y + \mathbf{6}_x + 8_y + 12_y\\
&\xrightarrow{1\text{ moves of } 6_x} 1_x + 4_y + \mathbf{7}_x + 8_y + 12_y\\
&\xrightarrow{\text{adjustment}} 1_x + 4_y + 7_y + \mathbf{9}_x + 12_y\\
&\xrightarrow{2\text{ moves of } 1_x} \mathbf{3}_x + 4_y + 7_y + \mathbf{9}_x + 12_y\\
&\xrightarrow{\text{adjustment}} 3_y + \mathbf{5}_x + 7_y + \mathbf{9}_x + 12_y =: \la.
\end{align*}
Here boldface marks the active $x$-part, and intermediate steps displaying equal numerical parts denote temporary configurations immediately resolved by adjustments. The output $\la = 3_y + 5_x + 7_y + 9_x + 12_y \in \sA_{3, 2}^{\II}$ is admissible and satisfies
\begin{align*}
|\be^{\II, (3, 2)}| + |\ga^1| + |\ga^2| = 18 + 10 + 8 = 36 = 3 + 5 + 7 + 9 + 12 = |\la|.
\end{align*}
\end{example}

\begin{remark}
In the proof of Lemma \ref{lem:sec2-psiII} and the subsequent proofs, we use the combinatorial framework described in \cite{FL241} (the ``base+increments'' machinery). This framework can also be found in the work of Kur\c{s}ung\"{o}z \cite{kur10, kur19, kur21, KS22}, and the same idea appeared earlier in \cite{AG95, bre80, war97}.
\end{remark}

By \eqref{sec2-3} and Lemma \ref{lem:sec2-psiII}, we have
\begin{align}
\sum_{\la\in \sA_{j, k}^{\II}} x^k y^j q^{|\la|} = \frac{x^k y^j q^{j^2+k^2/2+jk+k/2}}{(q; q)_j(q; q)_k}.
\end{align}
Combining this with \eqref{sec2-4} gives
\begin{align}\label{sec2-pair-sum}
\sum_{(\la, \mu)\in \bigcup_{j, k\geq 0}(\sA_{j, k}^{\II}\times \sB_{j, k}^{\II})} x^k y^j (-y)^{\ell(\mu)} q^{|\la|+|\mu|} = \text{LHS of }\eqref{sec2-1}.
\end{align}
We now define an involution to cancel all terms with non-empty $y$-parts.

\begin{lemma}\label{lem:varphi-II}
There exists an involution
\begin{align*}
\varphi^{\II} \colon \bigcup_{j, k\geq 0}(\sA_{j, k}^{\II}\times \sB_{j, k}^{\II}) &\longrightarrow \bigcup_{j, k\geq 0}(\sA_{j, k}^{\II}\times \sB_{j, k}^{\II}),\\
(\la, \mu) &\longmapsto (\hla, \hmu),
\end{align*}
such that $|\la| + |\mu| = |\hla| + |\hmu|$, $\ell(\la) + \ell(\mu) = \ell(\hla) + \ell(\hmu)$, and the number of $x$-parts is preserved. Moreover, when $(\la, \mu)\neq (\hla, \hmu)$, we have $\ell(\mu) = \ell(\hmu) \pm 1$.
\end{lemma}

\begin{proof}
Let $(\la, \mu) \in \sA_{j, k}^{\II}\times \sB_{j, k}^{\II}$ with $\la = \la_1 + \cdots + \la_n$, where $n = j+k$. Let $\mu_1$ denote the smallest part of $\mu$, with the convention $\mu_1 = +\infty$ if $\mu = \ep$. If $\la$ contains at least one $y$-part, let $p$ be the index of its smallest $y$-part and define the removal statistic
\begin{align*}
r = \la_p + n - p.
\end{align*}
If $\la$ contains no $y$-part, set $r = +\infty$. If $r = \mu_1 = +\infty$, we define $(\la, \mu)$ to be a fixed point. Otherwise, the involution is defined according to the relative order of $r$ and $s := \mu_1$.

\begin{description}
\item[Case I] $r < s$. In this case, delete $\la_p$ from $\la$, subtract $1$ from each subsequent part, and retain all labels:
\begin{align*}
\hla_m = \begin{cases}
\la_m, & \text{if } 1\leq m\leq p-1,\\
\la_{m+1} - 1, & \text{if } p\leq m\leq n-1.
\end{cases}
\end{align*}
Insert $r_y$ into $\mu$ as its new smallest part, setting $\hmu = r_y + \mu_1 + \mu_2 + \cdots$.

We verify that $(\hla, \hmu) \in \sA_{j-1, k}^{\II}\times \sB_{j-1, k}^{\II}$. Because $\la_p \geq p$, we have $r = \la_p + n - p \geq n$. The shortened partition $\hla$ has length $n-1$, so the lower bound required for parts in $\sB_{j-1, k}^{\II}$ is $(n-1)+1 = n$, which is satisfied by $r$. Since $r < \mu_1$, $\hmu$ is strictly increasing. In $\hla$, the part $\la_{p-1}$ (if present) is labeled $x$, so strictness $\hla_{p-1} < \hla_p$ suffices. If $\la_{p+1}$ exists, the original gap condition $\la_{p+1} - \la_p \geq 2$ gives $\hla_p - \hla_{p-1} = (\la_{p+1}-1) - \la_{p-1} \geq (\la_p+1) - \la_{p-1} \geq 2$. All subsequent gaps are unchanged, so $\hla \in \sA_{j-1, k}^{\II}$. The change in the weight of $\la$ is $-\la_p - (n-p) = -r$, which precisely offsets the weight added to $\mu$.

\item[Case II] $r \geq s$.
Here $s = \mu_1 < +\infty$. Define the nondecreasing sequence of threshold statistics
\begin{align*}
T_a = \la_a + n - a \quad (1\leq a\leq n), \qquad T_{n+1} = +\infty.
\end{align*}
Since $\la$ is strict, $\la_{a+1} - \la_a \geq 1$, which implies $T_{a+1} \geq T_a$. Let $a \in \{1, \dots, n+1\}$ be the smallest index such that $T_a \geq s$. We insert at position $a$ a new $y$-part of size
\begin{align*}
b = s - n + a - 1,
\end{align*}
add $1$ to each old part at positions $a, \dots, n$ (retaining their labels), and remove $s$ from $\mu$:
\begin{align*}
\hla_m = \begin{cases}
\la_m, & \text{if } 1\leq m\leq a-1,\\
b, & \text{if } m = a,\\
\la_{m-1} + 1, & \text{if } a+1\leq m\leq n+1.
\end{cases}
\end{align*}
When $a = n+1$, the part $b = s$ is appended at the end; this in particular covers the boundary case $\la = \ep$ ($n=0$).

We check that $(\hla, \hmu) \in \sA_{j+1, k}^{\II}\times \sB_{j+1, k}^{\II}$. Since $s \geq n+1$, we have $b \geq (n+1) - n + a - 1 = a \geq 1$. If $a > 1$, the minimality of $a$ gives $T_{a-1} < s$, so $\la_{a-1} < b$. Because the first old $y$-part has statistic $T_p = r \geq s$, we must have $a \leq p$, meaning that no $y$-part precedes position $a$; thus the part $\la_{a-1}$ is labeled $x$, and the strict inequality $\la_{a-1} < b$ is sufficient. If $a \leq n$, the inequality $T_a \geq s$ implies $(\la_a + 1) - b = T_a - s + 2 \geq 2$, fulfilling the required gap following the new $y$-part. All remaining parts of $\hmu$ are strictly greater than $s \geq n+1$, so they are at least $n+2 = (n+1)+1$. The weight added to $\la$ is $b + (n - a + 1) = s$, which cancels the weight removed from $\mu$.

\item[Involutive property]
After the insertion in Case II, the newly inserted part $b_y$ at position $a$ is the smallest $y$-part of $\hla$. Its removal statistic in the enlarged partition of length $n+1$ is
\begin{align*}
b + (n+1) - a = (s - n + a - 1) + n + 1 - a = s.
\end{align*}
Any subsequent $y$-part has new statistic $T_h + 1 > s$. Since the smallest part of $\hmu$ is strictly greater than $s$, the pair $(\hla, \hmu)$ satisfies $r_{\text{new}} = s < \hmu_1$, which places it squarely in Case I. Deletion of the part at position $a$ then precisely reverses the insertion.

Conversely, after a deletion at position $p$ with removed value $r$ in Case I, parts preceding position $p$ have new statistics $T_h - 1 < r$. If a part followed $\la_p$, its new statistic at position $p$ is $(\la_{p+1}-1) + (n-1) - p = T_{p+1} - 1 \geq r$, thanks to the gap $\la_{p+1} - \la_p \geq 2$. If no part followed, the terminal value $T_n = +\infty \geq r$ plays the same role. Thus the minimal index chosen in Case II is precisely the original position $p$. Reinsertion restores both $\la$ and $\mu$.
\end{description}

Each operation transfers one $y$-part between $\la$ and $\mu$, preserving total weight, total length, and the number of $x$-parts while changing $\ell(\mu)$ by $1$. This establishes that $\varphi^{\II}$ is a sign-reversing involution.
\end{proof}

Table \ref{tab:bijection_varphiII_w6} displays the complete action of $\varphi^{\II}$ on all pairs of total weight $6$.
\begin{table}[htbp]
\centering
\renewcommand{\arraystretch}{1.25}
\begin{tabular}{ccc|ccc|c}
\toprule
Case I & & Case II & Case I & & Case II & Fixed points\\
\midrule
$(6_y, \ep)$ & $\longleftrightarrow$ & $(\ep, 6_y)$ & $(1_y, 5_y)$ & $\longleftrightarrow$ & $(\ep, 1_y+5_y)$ & $(6_x, \ep)$\\
$(2_y, 4_y)$ & $\longleftrightarrow$ & $(\ep, 2_y+4_y)$ & $(1_y, 2_y+3_y)$ & $\longleftrightarrow$ & $(\ep, 1_y+2_y+3_y)$ & $(1_x+5_x, \ep)$\\
$(1_y+5_x, \ep)$ & $\longleftrightarrow$ & $(4_x, 2_y)$ & $(2_y+4_x, \ep)$ & $\longleftrightarrow$ & $(3_x, 3_y)$ & $(2_x+4_x, \ep)$\\
$(2_x+4_y, \ep)$ & $\longleftrightarrow$ & $(2_x, 4_y)$ & $(1_x+5_y, \ep)$ & $\longleftrightarrow$ & $(1_x, 5_y)$ & $(1_x+2_x+3_x, \ep)$\\
$(1_x+2_y, 3_y)$ & $\longleftrightarrow$ & $(1_x, 2_y+3_y)$ & $(1_y+5_y, \ep)$ & $\longleftrightarrow$ & $(4_y, 2_y)$ & \\
$(2_y+4_y, \ep)$ & $\longleftrightarrow$ & $(3_y, 3_y)$ & $(1_x+2_x+3_y, \ep)$ & $\longleftrightarrow$ & $(1_x+2_x, 3_y)$ & \\
\bottomrule
\end{tabular}
\vspace{0.5em}
\caption{The involution $\varphi^{\II}$ for total weight $6$.}
\label{tab:bijection_varphiII_w6}
\end{table}

By Lemma \ref{lem:varphi-II}, the non-fixed pairs cancel in pairs in the sum \eqref{sec2-pair-sum}. The fixed points are precisely the pairs with $r = \mu_1 = +\infty$, which are
\begin{align*}
\bigcup_{k\geq 0}\bigl(\sA_{0, k}^{\II}\times \{\ep\}\bigr).
\end{align*}
These correspond bijectively to strict partitions containing only $x$-labeled parts. Consequently,
\begin{align*}
\text{LHS of }\eqref{sec2-1} &= \sum_{(\la, \mu)\in \bigcup_{j, k\geq 0}(\sA_{j, k}^{\II}\times \sB_{j, k}^{\II})} x^k y^j (-y)^{\ell(\mu)} q^{|\la|+|\mu|}\\
&= \sum_{\la\in \bigcup_{k\geq 0}\sA_{0, k}^{\II}} x^k q^{|\la|} = \sum_{\la\in \sD} x^{\ell(\la)} q^{|\la|} = (-xq; q)_{\infty}.
\end{align*}
This completes the combinatorial proof of Theorem \ref{thm:2}.

\section{Combinatorial proof of Theorem \ref{thm:3}}\label{sec:3}

In this section, we consider the joint case: for $E\in \{0, 1\}$, we provide a combinatorial proof for the following three-variable partition identity:
\begin{align}
\sum_{i, j, k\geq 0}\frac{x^{j+k}q^{\frac{1}{2}i^2+3j^2+2k^2+2ij+2ik+4jk+\frac{1}{2}i+E(i+2j)}}{(q; q)_i(q^2; q^2)_j(q^2; q^2)_k}=(-q^{E+1}; q)_{\infty}(-xq^2; q^4)_{\infty}.\label{sec3-1}
\end{align}
Let $\sD^{\geq E+1}$ denote the set of strict partitions $\pi\in \sD$ in which each part of $\pi$ is at least $E+1$, and let $\sD^{(2, 4)}$ denote the set of strict partitions $\al\in \sD$ in which each part is congruent to $2$ modulo $4$. Then the right-hand side of \eqref{sec3-1} generates pairs of such partitions:
\begin{align}\label{sec3-4}
\text{RHS of }\eqref{sec3-1}=\sum_{(\al, \pi)\in \sD^{(2, 4)}\times \sD^{\geq E+1}}x^{\ell(\al)}q^{|\al|+|\pi|}.
\end{align}
Next, we focus our attention on the left-hand side of \eqref{sec3-1}, which can be decomposed combinatorially as follows:
\begin{align}\label{sec3-2}
\text{LHS of }\eqref{sec3-1}=&\sum_{j, k\geq 0}\frac{x^{j+k}q^{[2+\cdots +(4k-2)]+[(4k+2E+3)+\cdots +(4k+2E+6j-3)]}}{(q^2; q^2)_j(q^2; q^2)_k}\nonumber\\
&\qquad\qquad\qquad\qquad\qquad\times \sum_{i\geq 0}\frac{q^{(2j+2k+E+1)+\cdots +(2j+2k+E+i)}}{(q; q)_i}.
\end{align}
For fixed $j, k\geq 0$, let \(\sB^{\III_E}_{j, k}\) denote the set of strict partitions $\mu$ whose parts are all strictly greater than $2j+2k+E$. Then we have
\begin{align}
\sum_{\mu\in \sB^{\III_E}_{j, k}}q^{|\mu|}=\sum_{i\geq 0}\frac{q^{(2j+2k+E+1)+\cdots +(2j+2k+E+i)}}{(q; q)_i}.
\end{align}

For the remaining summation over $j, k$ in \eqref{sec3-2}, we define the base partition
\begin{align}
\beta^{\III_E, (j, k)}:=2+\cdots +(4k-2)+(4k+2E+3)+\cdots +(4k+2E+6j-3).
\end{align}
Direct computation shows that the total weight is $|\beta^{\III_E, (j, k)}|=3j^2+2k^2+4jk+2Ej$, the number of parts is $\ell(\beta^{\III_E, (j, k)})=j+k$, and the number of odd parts is $\o(\beta^{\III_E, (j, k)})=j$. Let $\sE$ denote the set of partitions into even parts and $\sE_{n}:=\{\la\in \sE: \ell(\la)\leq n\}$. Then the left-hand factor satisfies
\begin{align}\label{sec3-3}
\sum_{(\beta^{\III_E, (j, k)}, \ga^1, \ga^2)\in \{\beta^{\III_E, (j, k)}\}\times \sE_j\times \sE_k}x^{j+k}q^{|\beta^{\III_E, (j, k)}|+|\ga^1|+|\ga^2|}=\frac{x^{j+k}q^{3j^2+2k^2+4jk+2Ej}}{(q^2; q^2)_j(q^2; q^2)_k}.
\end{align}

On the other hand, let $\sA^{\III_E}_{j, k}$ denote the set of strict partitions $\la$ with $j$ odd parts and $k$ even parts satisfying the following initial and difference conditions:
\begin{enumerate}[(i)]
\item the smallest even part is at least $2$, and the smallest odd part is at least $3+2E$;
\item the difference between adjacent even parts is at least $4$, and the difference between adjacent odd parts is at least $6$;
\item if an odd part and an even part are adjacent, their difference is at least $5+2E$ when the larger part is odd, and at least $5-2E$ when the larger part is even.
\end{enumerate}
One can easily check that $\beta^{\III_E, (j, k)}$ is the partition with the smallest weight in $\sA^{\III_E}_{j, k}$. We now construct the first bijection of this section.

\begin{lemma}\label{lem:sec3-psiIII}
For $j, k\geq 0$ and $E\in\{0, 1\}$, there exists a weight-preserving bijection
\begin{align*}
\psi^{\III}: \{\beta^{\III_E, (j, k)}\}\times \sE_j\times \sE_k&\longrightarrow \sA^{\III_E}_{j, k}\\
(\beta^{\III_E, (j, k)}, \ga^1, \ga^2)&\longmapsto \la,
\end{align*}
such that $|\beta^{\III_E, (j, k)}|+|\ga^1|+|\ga^2|=|\la|$, $\ell(\beta^{\III_E, (j, k)})=\ell(\la)$, and $\o(\beta^{\III_E, (j, k)})=\o(\la)$.
\end{lemma}

\begin{proof}
The proof proceeds analogously to that of Lemma \ref{lem:sec2-psiII}. Here, we list only the four core operations that are different; the remaining verification details will not be repeated. The map transforms the base partition $\beta^{\III_E, (j, k)}$ into $\la\in \sA_{j, k}^{\III_E}$ by successively adding even increments prescribed by $\ga^1\in \sE_j$ and $\ga^2\in \sE_k$. The elementary forward and backward moves, together with the corresponding adjustments and normalizations on an even part $t$, are defined as follows:
\begin{description}
\item[Forward move] An elementary forward move on an even part $t$ transforms it into $t+2$, increasing its weight by $2$.

\item[Adjustment] If the adjacent part larger than $t$ is an odd part equal to $t+2E+5$, then after one forward move, $t$ becomes $t+2$, and the difference drops to $(t+2E+5)-(t+2)=2E+3 < 2E+5$, which violates condition (iii) of $\sA_{j, k}^{\III_E}$. To restore admissibility, we define the adjustment:
\begin{align*}
&(\text{parts}\leq \max(t-4, t+2E-5)), \mathbf{t+2}, t+2E+5, (\text{parts}\geq t+10)\\
\xri{\text{an adjustment}}{}&(\text{parts}\leq \max(t-4, t+2E-5)), t+2E+1,\mathbf{t+6} , (\text{parts}\geq t+10).
\end{align*}

\item[Backward move] A backward move transforms an even part $t\ge 4$ into $t-2$, decreasing its weight by $2$.

\item[Normalization] If the preceding part smaller than $t$ is an odd part equal to $t+2E-5$, then after a backward move $t\mapsto t-2$, the difference becomes $(t-2)-(t+2E-5)=3-2E < 5-2E$, violating condition (iii). We define the normalization:
\begin{align*}
&(\text{parts}\leq \max(t-10, t+2E-11)), t+2E-5, \mathbf{t-2}, (\text{parts}\geq t+4)\\
\xri{\text{a normalization}}{}&(\text{parts}\leq \max(t-10, t+2E-11)), \mathbf{t-6}, t+2E-1,(\text{parts}\geq t+4).
\end{align*}
\end{description}
The forward moves and backward moves, as well as the adjustments and normalizations, are mutually inverse.  
\end{proof}

\begin{remark}
To provide a complete, self-contained verification that the weight-generating function of $\sA^{\III_E}_{j, k}$ matches \eqref{sec3-3}, we also present an independent proof by parity-word enumeration. Let $n=j+k$, and arrange the parts of any $\la\in \sA^{\III_E}_{j, k}$ in increasing order: $\la_1 < \la_2 < \cdots < \la_n$. Let $w=w_1 w_2 \cdots w_n \in \{\mathrm{e}, \mathrm{o}\}^n$ denote the parity word of $\la$, containing $j$ odd symbols ($\mathrm{o}$) and $k$ even symbols ($\mathrm{e}$). For each word $w$, there is a unique componentwise minimal sequence $b(w)=(b_1,\ldots,b_n)$ satisfying the boundary and difference conditions of $\sA^{\III_E}_{j, k}$. Specifically, $b_1$ is $2$ if $w_1=\mathrm{e}$ and $3+2E$ if $w_1=\mathrm{o}$; for each $1\leq r < n$, the gap $b_{r+1}-b_r$ takes the minimal permissible difference: $4$ for $\mathrm{ee}$, $6$ for $\mathrm{oo}$, $5+2E$ for $\mathrm{eo}$, and $5-2E$ for $\mathrm{oe}$.

Every partition $\la\in \sA^{\III_E}_{j, k}$ with parity word $w$ is uniquely represented as
\[
\la_r=b_r(w)+2h_r\qquad(1\leq r\leq n), \qquad 0\leq h_1\leq h_2\leq\cdots\leq h_n.
\]
Indeed, $h_1\geq 0$ corresponds to the lower bound on the first part, and $h_{r+1}\geq h_r$ is equivalent to the difference condition between $\la_{r+1}$ and $\la_r$. Summing over all non-negative sequences $(h_r)$, the generating function for a fixed parity word $w$ is $q^{|b(w)|}/(q^2;q^2)_n$.

When $w_0=\mathrm{e}^k \mathrm{o}^j$ (all even parts precede all odd parts), the minimal partition is exactly $\beta^{\III_E,(j,k)}$, whose weight is $B=3j^2+2k^2+4jk+2Ej$. An adjacent swap $\mathrm{eo}\mapsto \mathrm{oe}$ at positions $(r, r+1)$ replaces the minimal differences around these positions: direct calculation shows that $b_r$ increases by $1+2E$, $b_{r+1}$ decreases by $2E-1$ (that is, changes by $1-2E$), and all subsequent coordinates $b_{r+2}, \dots, b_n$ remain completely unchanged. Consequently, each such adjacent transposition increases the minimal weight by $(1+2E)+(1-2E)=2$. Therefore,
\[
|b(w)| = B + 2\operatorname{inv}(w),
\]
where $\operatorname{inv}(w)$ is the number of pairs $(r, s)$ with $r<s$ such that $w_r=\mathrm{o}$ and $w_s=\mathrm{e}$. Applying the Gaussian binomial theorem, we obtain
\begin{align*}
\sum_{\la\in\sA^{\III_E}_{j,k}}q^{|\la|}
&=\frac{q^B}{(q^2;q^2)_{j+k}}\sum_{\substack{w\in \{\mathrm{o}^j, \mathrm{e}^k\}}} q^{2\operatorname{inv}(w)}
=\frac{q^B}{(q^2;q^2)_{j+k}}\genfrac{[}{]}{0pt}{}{j+k}{j}_{q^2}
=\frac{q^B}{(q^2;q^2)_j(q^2;q^2)_k}.
\end{align*}
This identity holds for all $j, k\geq 0$ (including $j=0$ or $k=0$). 
\end{remark}

Using \eqref{sec3-3} and Lemma \ref{lem:sec3-psiIII}, we immediately have
\begin{align}
\sum_{\la\in \sA_{j, k}^{\III_E}}x^{\ell(\la)}q^{|\la|}=\frac{x^{j+k}q^{3j^2+2k^2+4jk+2Ej}}{(q^2; q^2)_j(q^2; q^2)_k}.
\end{align}
Combining this with \eqref{sec3-2} shows that
\begin{align}
\sum_{(\la, \mu)\in \bigcup_{j, k\geq 0}(\sA_{j, k}^{\III_E}\times \sB_{j, k}^{\III_E})}x^{\ell(\la)}q^{|\la|+|\mu|}=\text{LHS of }\eqref{sec3-1}.
\end{align}
Our next step is to establish the second bijection from $\bigcup_{j, k\geq 0}(\sA_{j, k}^{\III_E}\times \sB_{j, k}^{\III_E})$ to $\sD^{(2, 4)}\times \sD^{\geq E+1}$, thereby completing the combinatorial proof of Theorem \ref{thm:3}.

\begin{lemma}\label{lem:sec3-varphiIII}
For $E\in \{0, 1\}$, there exists a bijection
\begin{align*}
\varphi^{\III}: \bigcup_{j, k\geq 0}(\sA_{j, k}^{\III_E}\times \sB_{j, k}^{\III_E})& \longrightarrow \sD^{(2, 4)}\times \sD^{\geq E+1}\\
(\la, \mu)&\longmapsto (\al, \pi),
\end{align*}
such that $|\la|+|\mu|=|\al|+|\pi|$ and $\ell(\la)=\ell(\al)$.
\end{lemma}

\begin{proof}
First, we describe the forward map \(\varphi^{\III}\). For a given pair \((\la, \mu)\in\sA_{j, k}^{\III_E}\times \sB_{j, k}^{\III_E}\), we set the initial state as \((\la^{(0)}, \mu^{(0)}):=(\la, \mu)\), and iterate the following stripping algorithm. At step \(i\geq 0\):
\begin{description}
\item[Step 1] Examine the parts of \(\la^{(i)}\) from left to right. If all parts are congruent to \(2\) modulo \(4\), the algorithm terminates and outputs \((\al, \pi):=(\la^{(i)}, \mu^{(i)})\). Otherwise, locate the first part that is not congruent to \(2\) modulo \(4\), and denote its position by \(m\). By minimality, when \(m>1\), the first \(m-1\) parts are all congruent to \(2\) modulo \(4\).

\item[Step 2] We strip \(\la^{(i)}\) according to the parity of \(\la_m^{(i)}\):
\begin{description}
\item[Case I] If \(\la_m^{(i)}\) is odd, subtract \(2E+1\) from \(\la_m^{(i)}\); subtract \(2\) from each of the parts \(\la_{m+1}^{(i)}, \dots, \la_{j+k}^{(i)}\) following position \(m\); and leave the first \(m-1\) parts unchanged, yielding \(\la^{(i+1)}\). The total value stripped from \(\la^{(i)}\) is
\[
o_m=2(j+k-m)+2E+1.
\]
Insert \(o_m\) into \(\mu^{(i)}\) and reorder, obtaining the new strict partition \(\mu^{(i+1)}\).

\item[Case II] If \(\la_m^{(i)}\equiv 0\pmod 4\), subtract \(2\) from each of the parts \(\la_m^{(i)}, \la_{m+1}^{(i)}, \dots, \la_{j+k}^{(i)}\) from position \(m\) onward, and leave the first \(m-1\) parts unchanged, yielding \(\la^{(i+1)}\). The total value stripped from \(\la^{(i)}\) is
\[
e_m=2(j+k-m)+2.
\]
Insert \(e_m\) into \(\mu^{(i)}\) and reorder, obtaining the new strict partition \(\mu^{(i+1)}\).
\end{description}

\item[Termination guarantee] If Case I occurs at position \(m\), \(\la_m^{(i)}\) becomes even after subtracting the odd integer \(2E+1\). If the resulting part is congruent to \(2\) modulo \(4\), position \(m\) is settled; if it is congruent to \(0\) modulo \(4\), position \(m\) remains the first nonconforming position at step $i+1$, immediately triggering Case II. Case II then subtracts \(2\) from \(\la_m\), changing it into a value congruent to \(2\) modulo \(4\). Once position \(m\) is congruent to \(2\) modulo \(4\), all subsequent stripping operations occur strictly at positions $m'>m$. Therefore, each position undergoes Case I at most once and Case II at most once, guaranteeing termination within at most \(2(j+k)\) steps.
\end{description}

We now verify that after each stripping step, \(\la^{(i+1)}\) satisfies the difference conditions of $\sA^{\III_E}$. Since parts before position $m$ are unchanged and parts after position $m$ are all reduced by $2$, the only differences that can change are the left difference \(\la_m^{(i+1)}-\la_{m-1}^{(i+1)}\), the right difference \(\la_{m+1}^{(i+1)}-\la_m^{(i+1)}\), and the lower bound on $\la_1$ (when $m=1$).
\begin{itemize}
\item \emph{Verification for Case I (odd stripping)}:
\begin{description}
\item[Left difference] Since \(\la_{m-1}^{(i)}\equiv 2\pmod 4\) is even and \(\la_m^{(i)}\) is odd, we have \(\la_m^{(i)}-\la_{m-1}^{(i)}\geq 5+2E\). Thus,
\[
\la_m^{(i+1)}-\la_{m-1}^{(i+1)}=(\la_m^{(i)}-(2E+1))-\la_{m-1}^{(i)}\geq 4.
\]
Since both parts are now even, the adjacent even-even difference condition ($\geq 4$) is satisfied.

\item[Right difference] The new difference is
\[
\la_{m+1}^{(i+1)}-\la_m^{(i+1)}=(\la_{m+1}^{(i)}-2)-(\la_m^{(i)}-(2E+1))=(\la_{m+1}^{(i)}-\la_m^{(i)})+2E-1.
\]
If \(\la_{m+1}^{(i)}\) is even, the original difference is at least \(5-2E\), so the new difference is at least \((5-2E)+2E-1=4\). If \(\la_{m+1}^{(i)}\) is odd, the original difference is at least \(6\), so the new difference is at least \(6+2E-1=5+2E\). Both satisfy the requirements of \(\sA^{\III_E}\).

\item[First part] If $m=1$ and $\la_1^{(i)}$ is odd, then $\la_1^{(i)}\geq 3+2E$. Hence $\la_1^{(i+1)}=\la_1^{(i)}-(2E+1)\geq 2$, which is an even integer satisfying the lower bound.
\end{description}

\item \emph{Verification for Case II (even stripping)}:
\begin{description}
\item[Left difference] Here \(\la_{m-1}^{(i)}\equiv 2\pmod 4\) and \(\la_{m}^{(i)}\equiv 0\pmod 4\). Since both are even and their difference cannot be a multiple of $4$, the condition that adjacent even parts differ by at least $4$ implies \(\la_{m}^{(i)}-\la_{m-1}^{(i)}\geq 6\). Therefore,
\[
\la_{m}^{(i+1)}-\la_{m-1}^{(i+1)}=(\la_{m}^{(i)}-2)-\la_{m-1}^{(i)}\geq 4.
\]

\item[Right difference] Since \(2\) is subtracted simultaneously from \(\la_m^{(i)}\) and \(\la_{m+1}^{(i)}\), the difference \(\la_{m+1}^{(i+1)}-\la_m^{(i+1)}=\la_{m+1}^{(i)}-\la_m^{(i)}\) and the parity of $\la_{m+1}$ are preserved. Hence the condition holds automatically.

\item[First part] If $m=1$ and \(\la_1^{(i)}\equiv 0\pmod 4\), then \(\la_1^{(i)}\geq 4\), so \(\la_1^{(i+1)}=\la_1^{(i)}-2\geq 2\).
\end{description}
\end{itemize}

Next, we check that all stripped parts are distinct and strictly smaller than the parts of $\mu$, ensuring \(\pi\in \sD^{\geq E+1}\). The stripped parts are
\[
o_m=2(j+k-m)+2E+1\in \{2E+1, 2E+3, \dots, 2(j+k)+2E-1\},
\]
which are all odd, and
\[
e_m=2(j+k-m)+2\in \{2, 4, \dots, 2(j+k)\},
\]
which are all even. Thus, stripped parts of opposite parities never coincide, and stripped parts of the same parity from distinct positions are strictly distinct. Furthermore, each position undergoes each case at most once. The maximum stripped part satisfies
\[
\max_{1\le m\le j+k}\{o_m, e_m\} = 2(j+k)+E.
\]
Since every part of $\mu\in \sB^{\III_E}_{j, k}$ is strictly greater than $2(j+k)+E$ and the smallest stripped part is at least $E+1$, the combined partition $\pi$ is strict and belongs to $\sD^{\geq E+1}$.

Finally, the total subtracted weight matches the sum of inserted parts, so $|\la|+|\mu|=|\al|+|\pi|$. Moreover, the operations preserve the number of parts: $\ell(\la)=\ell(\al)$.

Conversely, we construct the inverse map \((\varphi^{\III})^{-1}\). Given \((\al, \pi)\in \sD^{(2, 4)}\times \sD^{\geq E+1}\), set $l:=\ell(\al)=j+k$. Extract all parts strictly greater than $2l+E$ from $\pi$ to form $\mu\in \sB^{\III_E}_{j, k}$. The remaining parts of $\pi$ form a subset $R\subseteq \{E+1, \dots, 2l+E\}$. For each $r\in R$:
\begin{itemize}
\item if $r$ is even, setting $m=l-(r-2)/2$ uniquely recovers that Case II occurred at position $m$;
\item if $r$ is odd, setting $m=l-(r-(2E+1))/2$ uniquely recovers that Case I occurred at position $m$.
\end{itemize}
We reconstruct $\la$ by traversing $m$ from $l$ down to $1$ and modifying the current partition successively:
\begin{description}
\item[Step 1] Undo Case II first: if $2(l-m)+2\in R$, add $2$ to each part from position $m$ through $l$ of the current partition.
\item[Step 2] Undo Case I next: if $2(l-m)+2E+1\in R$, add $2E+1$ to the part at position $m$ and add $2$ to each subsequent part from position $m+1$ through $l$ of the current partition.
\end{description}
The verification of difference conditions for the inverse map mirrors the forward verification in reverse. When $m=1$ is completed, we recover $\la\in \sA_{j, k}^{\III_E}$. Since each forward step is inverted deterministically, $\varphi^{\III}$ is a well-defined bijection.
\end{proof}

We conclude this section with an illustrative example of the bijection \(\varphi^{\III}\).

\begin{example}
Set parameters \(E=0\) and \(j=k=2\), so that $l=j+k=4$. Let the initial input be
\[
(\la^{(0)}, \mu^{(0)}):=(\la, \mu)=(5+12+19+28, 12+17)\in \sA_{2, 2}^{\III_0}\times \sB_{2, 2}^{\III_0},
\]
with initial weight $|\la|+|\mu|=64+29=93$. We apply the stripping algorithm:
\begin{description}
\item[Stage 1] In \(\la^{(0)}\), the first part not congruent to \(2\) modulo \(4\) is \(\la_1^{(0)}=5\) (odd). Case I applies at $m=1$: subtract $1$ from $\la_1$ and $2$ from each subsequent part, yielding \(\la^{(1)}=4+10+17+26\). The stripped part is $o_1=2(4-1)+1=7$, so \(\mu^{(1)}=7+12+17\).

\item[Stage 2] In \(\la^{(1)}\), the first nonconforming part is \(\la_1^{(1)}=4\equiv 0\pmod 4\). Case II applies at $m=1$: subtract $2$ from all parts, yielding \(\la^{(2)}=2+8+15+24\). The stripped part is $e_1=2(4-1)+2=8$, so \(\mu^{(2)}=7+8+12+17\).

\item[Stage 3] In \(\la^{(2)}\), the first nonconforming part is \(\la_2^{(2)}=8\equiv 0\pmod 4\). Case II applies at $m=2$: subtract $2$ from parts $2, 3, 4$, yielding \(\la^{(3)}=2+6+13+22\). The stripped part is $e_2=2(4-2)+2=6$, so \(\mu^{(3)}=6+7+8+12+17\).

\item[Stage 4] In \(\la^{(3)}\), the first nonconforming part is \(\la_3^{(3)}=13\) (odd). Case I applies at $m=3$: subtract $1$ from $\la_3$ and $2$ from $\la_4$, yielding \(\la^{(4)}=2+6+12+20\). The stripped part is $o_3=2(4-3)+1=3$, so \(\mu^{(4)}=3+6+7+8+12+17\).

\item[Stage 5] In \(\la^{(4)}\), the first nonconforming part is \(\la_3^{(4)}=12\equiv 0\pmod 4\). Case II applies at $m=3$: subtract $2$ from parts $3, 4$, yielding \(\la^{(5)}=2+6+10+18\). The stripped part is $e_3=2(4-3)+2=4$, so \(\mu^{(5)}=3+4+6+7+8+12+17\).
\end{description}
All parts of \(\la^{(5)}\) are now congruent to \(2\) modulo \(4\), so the algorithm terminates. We obtain
\[
(\al, \pi)=\varphi^{\III}((\la, \mu))=(2+6+10+18, 3+4+6+7+8+12+17)\in \sD^{(2, 4)}\times \sD^{\geq 1}.
\]
The final weights satisfy $|\al|+|\pi|=36+57=93=|\la|+|\mu|$ and $\ell(\al)=4=\ell(\la)$. The inverse procedure retraces these stages in reverse order, uniquely recovering $(\la, \mu)$.
\end{example}

\section{Combinatorial proof of Theorem \ref{thm:4}}\label{sec:4}

In this section, we give a combinatorial proof of the following identity in standard form, from which the two identities in Theorem \ref{thm:4} can be obtained by simple algebraic computations:
\begin{align}\label{sec4-1}
\sum_{i,j,k\geq 0}
\frac{(-1)^jx^iy^{j+k}q^{i^2/2+3j^2+2k^2+2ij+2ik+4jk+k+i/2}}
{(q;q)_i(q^2;q^2)_j(q^2;q^2)_k}
=(-xq;q)_{\infty}.
\end{align}
More precisely, the substitutions used to recover those identities are
\begin{align*}
&\eqref{sec4-1}
\xri{\text{$q\mapsto q^2$ and \(x\mapsto xq^{-1}\)}}
\eqref{eq4-1},\qquad
\eqref{sec4-1}
\xri{\text{\(y\mapsto x^2\), \(-xq\mapsto x\) and \(q\mapsto q^2\)}}
\eqref{eq4-2}.
\end{align*}

The right-hand side of \eqref{sec4-1} has the same partition-theoretic
interpretation as the right-hand side of \eqref{sec1-strict}. We therefore
focus our attention on the left-hand side and decompose it as follows:
\begin{align}\label{sec4-2}
\text{LHS of }\eqref{sec4-1}
={}&\sum_{i\geq 0}\frac{x^iq^{i^2/2+i/2}}{(q;q)_i}
\sum_{j,k\geq 0}
\frac{y^kq^{(2i+3)+\cdots+(2i+4k-1)}}{(q^2;q^2)_k}
\frac{(-y)^jq^{(2i+4k+3)+\cdots+(2i+4k+6j-3)}}{(q^2;q^2)_j}.
\end{align}
Here and below, an empty sum appearing in the exponent is understood to be zero.

We first interpret the terms in the inner double sum combinatorially. For fixed $i\geq 0$,
let \(\sA_{i,k}^{\IV}\) be the set of strict partitions \(\la\) into \(k\)
odd parts whose smallest part is at least \(2i+3\) and whose adjacent parts
differ by at least \(4\). Let \(\sB_{i,j,k}^{\IV}\) be the set of strict
partitions \(\mu\) into \(j\) odd parts whose smallest part is at least
\(2i+4k+3\) and whose adjacent parts differ by at least \(6\).
Writing parts in increasing order, these two sets can be parameterized explicitly as
\begin{align*}
\sA_{i,k}^{\IV}
&:=\{\la\in\sD:\ell(\la)=k,\ 
\la_m=2t_m+2i+4m-1\text{ for }1\leq m\leq k,\
0\leq t_1\leq\cdots\leq t_k\},\\
\sB_{i,j,k}^{\IV}
&:=\{\mu\in\sD:\ell(\mu)=j,\
\mu_m=2s_m+2i+4k+6m-3\text{ for }1\leq m\leq j,\
0\leq s_1\leq\cdots\leq s_j\}.
\end{align*}
Indeed, subtracting the respective minimal possible parts leaves a
nondecreasing sequence of nonnegative even increments. Consequently,
\begin{align}\label{sec4-3}
\sum_{(\la,\mu)\in\sA_{i,k}^{\IV}\times\sB_{i,j,k}^{\IV}}
y^{\ell(\la)}(-y)^{\ell(\mu)}q^{|\la|+|\mu|}
={}&
\frac{y^kq^{(2i+3)+\cdots+(2i+4k-1)}}{(q^2;q^2)_k}
\frac{(-y)^jq^{(2i+4k+3)+\cdots+(2i+4k+6j-3)}}{(q^2;q^2)_j}.
\end{align}

For fixed \(i\geq 0\), we seek an involution
\(\varphi^{\IV'}\) on
\(\bigcup_{j,k\geq 0}
(\sA_{i,k}^{\IV}\times\sB_{i,j,k}^{\IV})\).
As in Section \ref{sec:1}, subtracting the common base \(2i\)
from every part reduces the construction to \(i=0\). We abbreviate
\(\sA_k^{\IV}=\sA_{0,k}^{\IV}\) and
\(\sB_{j,k}^{\IV}=\sB_{0,j,k}^{\IV}\).
It is therefore sufficient to construct the following sign-reversing involution.

\begin{lemma}\label{lem:sec4-varphiIV}
There exists an involution
\begin{align*}
\varphi^{\IV}:
\bigcup_{j,k\geq 0}(\sA_k^{\IV}\times\sB_{j,k}^{\IV})
&\ri
\bigcup_{j,k\geq 0}(\sA_k^{\IV}\times\sB_{j,k}^{\IV})\\
(\la,\mu)&\mapsto(\hla,\hmu),
\end{align*}
such that $|\la|+|\mu|=|\hla|+|\hmu|$ and
$\ell(\la)+\ell(\mu)=\ell(\hla)+\ell(\hmu)$. Moreover, when
$(\la,\mu)\neq(\hla,\hmu)$, we have
$\ell(\mu)=\ell(\hmu)\pm 1$.
\end{lemma}

\begin{proof}
First, we set \(\varphi^{\IV}(\ep,\ep)=(\ep,\ep)\) as the unique fixed point.
For any nonempty partition pair
\((\la,\mu)\in\sA_k^{\IV}\times\sB_{j,k}^{\IV}\) (so that \(j+k\geq 1\)),
exactly one of the following two cases applies.

\begin{description}
\item[Case I]
Suppose \(k>0\) and either \(j=0\) or \(t_k\geq s_j+j\).
Delete the largest part \(2t_k+4k-1\) of \(\la\) to obtain \(\hla\).
Subtract \(4\) from each part of \(\mu\), and append the new largest part
\[
\hmu_{j+1}=2t_k+4k+4j-1
\]
to obtain \(\hmu\).

We verify that \((\hla,\hmu)\in\sA_{k-1}^{\IV}\times\sB_{j+1,k-1}^{\IV}\).
Clearly, \(\hla\in\sA_{k-1}^{\IV}\). For \(1\leq m\leq j\), the corresponding part of \(\hmu\) is
\[
\mu_m-4 = 2s_m+4(k-1)+6m-3;
\]
thus its parameters remain \(s_1,\ldots,s_j\). The appended part can be rewritten as
\[
\hmu_{j+1} = 2(t_k-j)+4(k-1)+6(j+1)-3.
\]
If \(j>0\), the case hypothesis gives \(t_k-j\geq s_j\geq 0\); if \(j=0\), it yields the nonnegative parameter \(t_k\geq 0\). Hence the complete parameter sequence of \(\hmu\) remains nonnegative and nondecreasing. In particular, when \(j>0\), the required gap at the newly appended part satisfies
\[
\hmu_{j+1}-\hmu_j = 2(t_k-s_j)+6-2j = 2(t_k-s_j-j)+6\geq 6.
\]
When \(k=1\), the remaining partition \(\hla=\ep\) belongs to \(\sA_0^{\IV}\).

The deleted part of \(\la\) has weight \(2t_k+4k-1\). Subtracting \(4\) from the \(j\) old parts of \(\mu\) and appending a part of weight \(2t_k+4k+4j-1\) preserves the total weight:
\[
-(2t_k+4k-1) - 4j + (2t_k+4k+4j-1) = 0.
\]
The numbers of parts change from \((k,j)\) to \((k-1,j+1)\), so their sum is strictly preserved.

\item[Case II]
Suppose \(j>0\) and either \(k=0\) or \(t_k<s_j+j\).
Delete the largest part \(2s_j+4k+6j-3\) of \(\mu\), and add \(4\) to each remaining part of \(\mu\) to obtain \(\hmu\). Append
\[
\hla_{k+1}=2s_j+4k+2j+1
\]
to \(\la\) to obtain \(\hla\).

For \(1\leq m\leq j-1\), we have
\[
\hmu_m=\mu_m+4 = 2s_m+4(k+1)+6m-3.
\]
Thus \(\hmu\in\sB_{j-1,k+1}^{\IV}\), including when \(j=1\) and \(\hmu=\ep\). The appended part of \(\hla\) has the required form
\[
\hla_{k+1} = 2(s_j+j-1)+4(k+1)-1.
\]
Its parameter \(s_j+j-1\) is nonnegative since \(j\geq 1\) and \(s_j\geq 0\). If \(k>0\), the case condition and integrality imply \(t_k\leq s_j+j-1\); hence the parameter sequence of \(\hla\) remains nondecreasing. Equivalently, its final gap satisfies
\[
\hla_{k+1}-\hla_k = 2(s_j-t_k)+2j+2 = 2(s_j+j-1-t_k)+4\geq 4.
\]
When \(k=0\), there is no preceding part to check. Therefore, \((\hla,\hmu)\in\sA_{k+1}^{\IV}\times\sB_{j-1,k+1}^{\IV}\).

The appended part of \(\hla\) has weight \(2s_j+4k+2j+1\), which is \(4(j-1)\) less than the deleted part of \(\mu\). Adding \(4\) to each of the remaining \(j-1\) parts of \(\mu\) compensates exactly for this difference. Total weight is therefore preserved, as is the total number of parts, which changes from \((k,j)\) to \((k+1,j-1)\).
\end{description}

It remains to check that the two operations are mutually inverse.
After Case I, the parameter of the newly appended largest part of \(\hmu\) is \(\hat{s}_{j+1}=t_k-j\). If \(k=1\), the image has \(\hla=\ep\) and \(\ell(\hmu)=j+1\geq 1\), so it falls directly into Case II. If \(k>1\), the largest parameter of \(\hla\) is \(\hat{t}_{k-1}=t_{k-1}\), and
\[
\hat{t}_{k-1}=t_{k-1}\leq t_k < (t_k-j)+(j+1) = \hat{s}_{j+1}+(j+1).
\]
Thus the image again satisfies the condition of Case II. Applying Case II deletes the newly appended part, adds \(4\) back to each of the first \(j\) parts of \(\hmu\), and appends a part to \(\hla\) with parameter
\[
\hat{s}_{j+1}+(j+1)-1 = (t_k-j)+j = t_k,
\]
which completely restores \((\la,\mu)\).

Conversely, after Case II, the parameter of the newly appended largest part of \(\hla\) is \(\hat{t}_{k+1}=s_j+j-1\). If \(j=1\), the image has \(\hmu=\ep\) and \(\ell(\hla)=k+1\geq 1\), which falls directly into Case I. If \(j>1\), its largest \(\hmu\)-parameter is \(\hat{s}_{j-1}=s_{j-1}\), and
\[
\hat{s}_{j-1}+(j-1) = s_{j-1}+j-1 \leq s_j+j-1 = \hat{t}_{k+1}.
\]
It therefore falls into Case I. Applying Case I deletes the newly appended part of \(\hla\), subtracts \(4\) from each of the \(j-1\) parts of \(\hmu\), and appends a part to \(\hmu\) with parameter
\[
\hat{t}_{k+1}-(j-1) = (s_j+j-1)-(j-1) = s_j,
\]
which again completely restores \((\la,\mu)\).

Every nonempty pair belongs to exactly one of the two cases, and each case is mapped to the other. Consequently, \(\varphi^{\IV}\) is an involution. Its only fixed point is \((\ep,\ep)\), and the change in the number of \(\mu\)-parts is precisely \(\pm 1\).
\end{proof}

The involution preserves the total weight and the total number of parts while reversing the sign \((-1)^{\ell(\mu)}\). Since \((\ep,\ep)\) is its unique fixed point, we have
\begin{align*}
\sum_{(\la,\mu)\in\bigcup_{j,k\geq 0}(\sA_k^{\IV}\times\sB_{j,k}^{\IV})}
y^{\ell(\la)}(-y)^{\ell(\mu)}q^{|\la|+|\mu|} = 1.
\end{align*}
Adding \(2i\) to every part preserves the exact same cancellation: each paired pair has the same total number of parts and therefore receives the same additional weight \(2i(\ell(\la)+\ell(\mu))\). It follows from \eqref{sec4-2} that
\begin{align*}
\text{LHS of }\eqref{sec4-1}
=\sum_{i\geq 0}\frac{x^iq^{1+2+\cdots+i}}{(q;q)_i}
=\sum_{\la\in\sD}x^{\ell(\la)}q^{|\la|}
=(-xq;q)_{\infty},
\end{align*}
which completes the combinatorial proof of \eqref{sec4-1}.

We conclude this section with Table \ref{tab:bijection_varphiIV}, which lists all partition pairs of total weight \(31\) under the involution \(\varphi^{\IV}\).

\begin{table}[htbp]
\centering
\renewcommand{\arraystretch}{1.25}
\begin{tabular}{ccc|ccc}
\toprule
Case I & & Case II & Case I & & Case II \\
\midrule
$(31,\ep)$ & $\longleftrightarrow$ & $(\ep,31)$
& $(9,7+15)$ & $\longleftrightarrow$ & $(\ep,3+11+17)$ \\
$(11,7+13)$ & $\longleftrightarrow$ & $(\ep,3+9+19)$
& $(3+13,15)$ & $\longleftrightarrow$ & $(3,11+17)$ \\
$(3+15,13)$ & $\longleftrightarrow$ & $(3,9+19)$
& $(5+13,13)$ & $\longleftrightarrow$ & $(5,9+17)$ \\
$(7+11,13)$ & $\longleftrightarrow$ & $(7,9+15)$
& $(3+17,11)$ & $\longleftrightarrow$ & $(3,7+21)$ \\
$(5+15,11)$ & $\longleftrightarrow$ & $(5,7+19)$
& $(7+13,11)$ & $\longleftrightarrow$ & $(7,7+17)$ \\
$(3+7+21,\ep)$ & $\longleftrightarrow$ & $(3+7,21)$
& $(3+9+19,\ep)$ & $\longleftrightarrow$ & $(3+9,19)$ \\
$(3+11+17,\ep)$ & $\longleftrightarrow$ & $(3+11,17)$
& $(5+9+17,\ep)$ & $\longleftrightarrow$ & $(5+9,17)$ \\
$(5+11+15,\ep)$ & $\longleftrightarrow$ & $(5+11,15)$
& & & \\
\bottomrule
\end{tabular}
\vspace{0.5em}
\caption{The involution $\varphi^{\mathrm{IV}}$ for total weight $31$.}
\label{tab:bijection_varphiIV}
\end{table}

\section{Combinatorial proof of Theorem \ref{thm:5}}\label{sec:5}

First, making the parameter replacements \(x\mapsto -xq^3\) followed by \(q^2\mapsto q\) in \eqref{eq5-1}, and \(x\mapsto -xq\) followed by \(q^2\mapsto q\) in \eqref{eq5-2}, we obtain
\begin{align}
\sum_{i, j, k\geq 0}\frac{(-1)^{j+k}x^{i+2j+k}q^{i^2/2+3j^2+2k^2+2ij+2ik+4jk+i/2-k}}{(q; q)_i(q^2; q^2)_j(q^2; q^2)_k}&=(-xq^2; q^2)_{\infty},\label{sec5-1}\\
\sum_{i, j, k\geq 0}\frac{(-1)^{j+k}x^{i+2j+k}q^{i^2/2+3j^2+2k^2+2ij+2ik+4jk+i/2}}{(q; q)_i(q^2; q^2)_j(q^2; q^2)_k}&=(-xq; q^2)_{\infty}.\label{sec5-2}
\end{align}
Thus, the primary objective of this section is to provide a unified combinatorial proof of the following identity: for each \(E\in \{0, 1\}\),
\begin{align}
\sum_{i, j, k\geq 0}\frac{(-1)^{j+k}x^{i+2j+k}q^{i^2/2+3j^2+2k^2+2ij+2ik+4jk+i/2+(E-1)k}}{(q; q)_i(q^2; q^2)_j(q^2; q^2)_k}&=(-xq^{2-E}; q^2)_{\infty}.\label{sec5-3}
\end{align}
We focus on the series on the left-hand side of \eqref{sec5-3}. First, we decompose the summation into
\begin{align}\label{sec5-4}
\text{LHS of }\eqref{sec5-3}&=\sum_{k\geq 0}\frac{(-1)^kx^kq^{2k^2+(E-1)k}}{(q^2; q^2)_k}\nonumber\\
&\qquad\times\sum_{i\geq 0}\frac{x^{i}q^{(2k+1)+\cdots +(2k+i)}}{(q; q)_i}\sum_{j\geq 0}\frac{(-1)^jx^{2j}q^{(4k+2i+3)+\cdots +(4k+2i+6j-3)}}{(q^2; q^2)_j}.
\end{align}
To interpret the inner double summation combinatorially, we define two families of partitions with distinct parts, denoted by $\sA_{i, k}^{\V}$ and $\sB_{i, j, k}^{\V}$, as follows:
\begin{align*}
\sA_{i, k}^{\V}&:=\left\{\la\in \sD: \ell(\la)=i, \la_m=2k+a_m+m\text{ for }1\leq m\leq i, 0\leq a_1\leq \cdots \leq a_i\right\},\\
\sB_{i, j, k}^{\V}&:=\left\{\mu\in\sD: \ell(\mu)=j, \mu_m=4k+2i+2b_m+6m-3\text{ for }1\leq m\leq j, 0\leq b_1\leq \cdots \leq b_j\right\}.
\end{align*} 
With these definitions, the generating function over $\sA_{i, k}^{\V}\times \sB_{i, j, k}^{\V}$ satisfies
\begin{align}\label{sec5-5}
\sum_{(\la, \mu)\in \sA_{i, k}^{\V}\times \sB_{i, j, k}^{\V}}(-1)^{\ell(\mu)}x^{\ell(\la)+2\ell(\mu)}q^{|\la|+|\mu|}=\frac{x^{i}q^{2ik+i^2/2+i/2}}{(q; q)_i}\cdot\frac{(-1)^jx^{2j}q^{4jk+2ij+3j^2}}{(q^2; q^2)_j}.
\end{align}
Recall that the set \(\sR_{i, k}\), introduced in Section \ref{sec:1}, is given by
\begin{align*}
\sR_{i, k}&:=\left\{\ga\in \sD:
\ell(\ga)=i, \ga_m=2k+r_{m}+m \text{ for }1\leq m\leq i,
0\leq r_1<\cdots <r_{i} \right\}.
\end{align*}
Next, for each fixed \(k\geq 0\), we construct a sign-reversing and weight-preserving involution on \(\bigcup_{i, j\geq 0}(\sA_{i, k}^{\V}\times \sB_{i, j, k}^{\V})\) whose fixed-point set is precisely \(\bigcup_{i\geq 0}(\sR_{i, k}\times \{\ep\})\).

\begin{lemma}\label{lem:sec5-varphiV}
For any fixed $k\geq 0$, there exists an involution 
\begin{align*}
\varphi^{\V}:\bigcup_{i, j\geq 0}(\sA_{i, k}^{\V}\times \sB_{i, j, k}^{\V})&\longrightarrow \bigcup_{i, j\geq 0}(\sA_{i, k}^{\V}\times \sB_{i, j, k}^{\V})\\
(\la, \mu)&\mapsto (\hla, \hmu),
\end{align*}
such that $|\la|+|\mu|=|\hla|+|\hmu|$ and $\ell(\la)+2\ell(\mu)=\ell(\hla)+2\ell(\hmu)$. Moreover, whenever $(\la, \mu)\neq (\hla, \hmu)$, we have $\ell(\mu)=\ell(\hmu)\pm 1$.
\end{lemma}

\begin{proof}
Let $(\lambda, \mu) \in \mathcal{A}_{i, k}^{\mathrm{V}} \times \mathcal{B}_{i, j, k}^{\mathrm{V}}$ with coordinate vectors $a = (a_1, \dots, a_i)$ and $b = (b_1, \dots, b_j)$. We define the index of the first adjacent equality in $a$ by
$$m := \min\left(\{ t \in \{1, \dots, i-1\} : a_t = a_{t+1} \} \cup \{ \infty \}\right),$$
with the convention that $a_{\infty} := \infty$. When $j = 0$ (so that $\mu = \ep$), we set $b_1 := \infty$. If $m = \infty$ and $b_1 = \infty$, then $\mu = \ep$ (i.e., $j = 0$) and the sequence $a$ satisfies $0 \le a_1 < a_2 < \dots < a_i$. Setting $r_t := a_t$ for each $1\leq t\leq i$, we see that $\lambda \in \mathscr{R}_{i, k}$. In this case, we define $\varphi^{\mathrm{V}}(\lambda, \ep) := (\lambda, \ep)$.

We divide the remaining configurations into two mutually exclusive cases.
\begin{description}
\item[Case I] $a_m\geq b_1$. In this case, $b_1 < \infty$, which implies $j \ge 1$. We define the insertion index
$$p := \min\left( \{ t \in \{1, \dots, i\} : a_t \ge b_1 \} \cup \{ i+1 \} \right).$$
By the minimality of $m$, the prefix $a_1 < a_2 < \dots < a_{p-1} < b_1$ is strictly increasing (which is empty if $p=1$). We define $(\hat{\lambda}, \hat{\mu})$ via their transformed coordinate vectors $\hat{a}$ and $\hat{b}$:
\begin{align*}
\hat{a} &:= (a_1, \dots, a_{p-1}, b_1, b_1, a_p, \dots, a_i) \in \mathbb{N}_{0}^{i+2}, \\
\hat{b} &:= (b_2 + 1, b_3 + 1, \dots, b_j + 1) \in \mathbb{N}_{0}^{j-1}.
\end{align*}
In terms of partition parts, $\hat{\lambda}$ has length $i + 2$ with
$$\hat{\lambda}_t = 
\begin{cases}
\lambda_t, & 1 \le t \le p-1, \\
2k + b_1 + p, & t = p, \\
2k + b_1 + p + 1, & t = p+1, \\
\lambda_{t-2} + 2, & p+2 \le t \le i+2,
\end{cases}$$
and $\hat{\mu}$ has length $j - 1$ with parts $\hat{\mu}_s = \mu_{s+1}$ for all $1 \le s \le j-1$.

\item[Case II] $a_m<b_1$. In this case, $a_m < \infty$, which implies $m \le i-1$ and $a_m = a_{m+1}$. We define $(\hat{\lambda}, \hat{\mu})$ via the transformed coordinate vectors
\begin{align*}
\hat{a} &:= (a_1, \dots, a_{m-1}, a_{m+2}, \dots, a_i) \in \mathbb{N}_{0}^{i-2}, \\
\hat{b} &:= (a_m, b_1 - 1, b_2 - 1, \dots, b_j - 1) \in \mathbb{N}_{0}^{j+1}.
\end{align*}
In terms of partition parts, $\hat{\lambda}$ has length $i - 2$ with
$$\hat{\lambda}_t = 
\begin{cases}
\lambda_t, & 1 \le t \le m-1, \\
\lambda_{t+2} - 2, & m \le t \le i-2,
\end{cases}$$
and $\hat{\mu}$ has length $j + 1$ with parts
$$\hat{\mu}_1 = 4k + 2a_m + 2i - 1, \qquad \hat{\mu}_{s+1} = \mu_s \quad (1 \le s \le j).$$
\end{description}
We now systematically verify the required properties of $\varphi^{\mathrm{V}}$.
\begin{description}
\item[Well-definedness] In Case I, we have $\hat{a}_p = \hat{a}_{p+1} = b_1$, along with $a_{p-1} < b_1$ (when $p \ge 2$) and $b_1 \le a_p$ (when $p \le i$). Thus $0 \le \hat{a}_1 \le \dots \le \hat{a}_{i+2}$, which ensures $\hat{\lambda} \in \sA_{i+2, k}^{\mathrm{V}}$. For $\hat{b}$, since $0 \le b_1 \le b_2 \le \dots \le b_j$, we have $0 \le b_2 + 1 \le \dots \le b_j + 1$, so $\hat{\mu} \in \sB_{i+2, j-1, k}^{\V}$. Its parts satisfy
\begin{align*}
\hat{\mu}_s &= 4k + 2(i+2) + 2(b_{s+1}+1) + 6s - 3 \\
&= 4k + 2i + 2b_{s+1} + 6(s+1) - 3 = \mu_{s+1},
\end{align*}
showing that all remaining parts of $\mu$ are preserved identically.

In Case II, the minimality of $m$ implies $a_1 < \dots < a_{m-1} < a_m = a_{m+1} \le a_{m+2}$, so $a_{m-1} < a_{m+2}$ (when $m \ge 2$ and $m \le i-2$). Thus $0 \le \hat{a}_1 \le \dots \le \hat{a}_{i-2}$, which ensures $\hat{\lambda} \in \sA_{i-2, k}^{\V}$. For $\hat{b}$, if $j = 0$, we have $\hat{b} = (a_m) \in \mathbb{N}_0^1$. If $j \ge 1$, the integer inequality $a_m < b_1$ implies $a_m \le b_1 - 1$, and hence $0 \le \hat{b}_1 = a_m \le \hat{b}_2 = b_1 - 1 \le \dots \le \hat{b}_{j+1} = b_j - 1$, placing $\hat{\mu} \in \sB_{i-2, j+1, k}^{\V}$. Furthermore, for each $s \ge 1$,
$$\hat{\mu}_{s+1} = 4k + 2(i-2) + 2(b_s - 1) + 6(s+1) - 3 = 4k + 2i + 2b_s + 6s - 3 = \mu_s,$$
so all existing parts are preserved. The newly formed first part is $\hat{\mu}_1 = 4k + 2(i-2) + 2a_m + 3 = 4k + 2a_m + 2i - 1$. When $j \ge 1$,
\begin{align*}
\mu_1 - \hat{\mu}_1 &= (4k + 2i + 2b_1 + 3) - (4k + 2a_m + 2i - 1) \\
&= 2(b_1 - a_m) + 4 \ge 2(1) + 4 = 6,
\end{align*}
which guarantees that $\hat{\mu}_1 < \mu_1$ with a gap of at least $6$.

\item[Weight and statistic invariance] In Case I, the weight gained by $\lambda$ is
\begin{align*}
|\hat{\lambda}| - |\lambda| &= \hat{\lambda}_p + \hat{\lambda}_{p+1} + 2(i - p + 1) \\
&= (2k + b_1 + p) + (2k + b_1 + p + 1) + 2i - 2p + 2 \\
&= 4k + 2i + 2b_1 + 3 = \mu_1.
\end{align*}
Since $|\hat{\mu}| - |\mu| = -\mu_1$, the total weight is strictly invariant: $|\hat{\lambda}| + |\hat{\mu}| = |\lambda| + |\mu|$. The lengths satisfy $\ell(\hat{\lambda}) = i + 2$ and $\ell(\hat{\mu}) = j - 1$, so $\ell(\hat{\lambda}) + 2\ell(\hat{\mu}) = i + 2 + 2(j - 1) = i + 2j$.

In Case II, the weight lost by $\lambda$ is
\begin{align*}
|\lambda| - |\hat{\lambda}| &= \lambda_m + \lambda_{m+1} + 2(i - m - 1) \\
&= (2k + a_m + m) + (2k + a_m + m + 1) + 2i - 2m - 2 \\
&= 4k + 2a_m + 2i - 1 = \hat{\mu}_1.
\end{align*}
Since $|\hat{\mu}| - |\mu| = \hat{\mu}_1$, the total weight is invariant: $|\hat{\lambda}| + |\hat{\mu}| = |\lambda| + |\mu|$. The lengths satisfy $\ell(\hat{\lambda}) = i - 2$ and $\ell(\hat{\mu}) = j + 1$, which gives $\ell(\hat{\lambda}) + 2\ell(\hat{\mu}) = i - 2 + 2(j + 1) = i + 2j$.

\item[Involution property ($\varphi^{\V}\circ\varphi^{\V}=\mathrm{id}$)] Suppose $(\lambda, \mu)$ falls into Case I. In $\hat{a}$, all entries strictly preceding index $p$ are strictly increasing, and $\hat{a}_p = \hat{a}_{p+1} = b_1$. Hence, the first adjacent equality index in $\hat{a}$ is $\hat{m} = p$, with value $\hat{a}_{\hat{m}} = b_1$. The first coordinate of $\hat{b}$ is $\hat{b}_1 = b_2 + 1 > b_1 = \hat{a}_{\hat{m}}$ (if $j \ge 2$) or $\hat{b}_1 = \infty > \hat{a}_{\hat{m}}$ (if $j = 1$). In either scenario, $\hat{a}_{\hat{m}} < \hat{b}_1$, so $(\hat{\lambda}, \hat{\mu})$ unconditionally falls into Case II. Applying Case II removes the two entries at positions $\hat{m}, \hat{m}+1$ (recovering the vector $a$) and sets the new $b$-coordinates to $(\hat{a}_{\hat{m}}, \hat{b}_1 - 1, \dots) = (b_1, (b_2+1)-1, \dots) = b$. Therefore, $\varphi^{\V}(\hat{\lambda}, \hat{\mu}) = (\lambda, \mu)$.

Now suppose $(\lambda, \mu)$ falls into Case II. In $\hat{a}$, there are no adjacent equal pairs before index $m$. If $m \le i-3$ and $a_{m+2} = a_{m+3}$, the first adjacent equality index satisfies $\hat{m} \ge m$ with $\hat{a}_{\hat{m}} \ge a_{m+2} \ge a_m$; if no adjacent equal pair exists in $\hat{a}$, we have $\hat{m} = \infty$. In either case, $\hat{a}_{\hat{m}} \ge a_m = \hat{b}_1$, so $(\hat{\lambda}, \hat{\mu})$ unconditionally falls into Case I. In Case I, the insertion index is $\min\{ t : \hat{a}_t \ge \hat{b}_1 \} = m$ because $\hat{a}_{m-1} = a_{m-1} < a_m = \hat{b}_1 \le a_{m+2} = \hat{a}_m$ (when $m \ge 2$ and $m \le i-2$; if $m=1$, the lower inequality is vacuous; if $m = i-1$, the insertion occurs after $\hat{a}_{i-2}$). Thus, two copies of $\hat{b}_1 = a_m$ are inserted back at positions $m, m+1$ (recovering $a$), and the $b$-coordinates become $(\hat{b}_2 + 1, \dots) = ((b_1-1)+1, \dots) = b$. Hence $\varphi^{\V}(\hat{\lambda}, \hat{\mu}) = (\lambda, \mu)$.
\end{description}
Consequently, $\varphi^{\V}$ is an exact involution whose fixed points are precisely $\bigcup_{i \ge 0} (\sR_{i, k} \times \{ \ep \})$, completing the proof.
\end{proof}

To illustrate the mechanism of $\varphi^{\V}$, Table \ref{tab:bijection_varphiV} displays the complete classification of all $34$ configurations of total weight $14$ when $k=0$, consisting of $11$ pairs and $12$ fixed points.

\begin{table}[htbp]
\centering
\renewcommand{\arraystretch}{1.25}
\begin{tabular}{ccc|c}
\toprule
Case I & & Case II & Fixed Points \\
\midrule
$(9, 5)$ & $\longleftrightarrow$ & $(1+2+11, \ep)$ & $(14, \ep)$ \\
$(7, 7)$ & $\longleftrightarrow$ & $(2+3+9, \ep)$ & $(1+13, \ep)$ \\
$(5, 9)$ & $\longleftrightarrow$ & $(3+4+7, \ep)$ & $(2+12, \ep)$ \\
$(3, 11)$ & $\longleftrightarrow$ & $(3+5+6, \ep)$ & $(3+11, \ep)$ \\
$(1, 13)$ & $\longleftrightarrow$ & $(1+6+7, \ep)$ & $(4+10, \ep)$ \\
$(1+6, 7)$ & $\longleftrightarrow$ & $(1+2+3+8, \ep)$ & $(5+9, \ep)$ \\
$(2+5, 7)$ & $\longleftrightarrow$ & $(1+2+4+7, \ep)$ & $(6+8, \ep)$ \\
$(3+4, 7)$ & $\longleftrightarrow$ & $(1+2+5+6, \ep)$ & $(1+3+10, \ep)$ \\
$(1+4, 9)$ & $\longleftrightarrow$ & $(1+3+4+6, \ep)$ & $(1+4+9, \ep)$ \\
$(2+3, 9)$ & $\longleftrightarrow$ & $(2+3+4+5, \ep)$ & $(1+5+8, \ep)$ \\
$(\ep, 3+11)$ & $\longleftrightarrow$ & $(1+2, 11)$ & $(2+4+8, \ep)$ \\
& & & $(2+5+7, \ep)$ \\
\bottomrule
\end{tabular}
\vspace{0.5em}
\caption{The involution $\varphi^{\mathrm{V}}$ for total weight $14$ when $k=0$.}
\label{tab:bijection_varphiV}
\end{table}

Applying the involution $\varphi^{\V}$ established in Lemma \ref{lem:sec5-varphiV} to the summation in \eqref{sec5-5}, all non-fixed pairs cancel mutually, leaving
\begin{align}
\sum_{(\la, \mu)\in \bigcup_{i, j\geq 0}(\sA_{i, k}^{\V}\times \sB_{i, j, k}^{\V})}(-1)^{\ell(\mu)}x^{\ell(\la)+2\ell(\mu)}q^{|\la|+|\mu|}=\sum_{\ga\in \bigcup_{i\geq 0}\sR_{i, k}}x^{\ell(\ga)}q^{|\ga|}=\sum_{i\geq 0}\frac{x^iq^{i^2+2ik}}{(q; q)_i}.
\end{align}
Substituting this reduction into \eqref{sec5-4}, the target identity \eqref{sec5-3} reduces to
\begin{align}\label{sec5-reduced}
\sum_{i, k\geq 0}\frac{(-1)^kx^{i+k}q^{i^2+2k^2+2ik+(E-1)k}}{(q; q)_i(q^2; q^2)_k}=(-xq^{2-E}; q^2)_{\infty}.
\end{align}
A purely combinatorial proof of the case $E=1$ in \eqref{sec5-reduced} was established via a sign-reversing involution in \cite[Eq. (1.3)]{FL241}. We also combinatorially proved a parameterized identity in \cite[Eq. (1.7)]{FL241}
\begin{align*}
\sum_{i,j\ge 0} \frac{q^{i^2+2ij+2j^2-i-j}}{(q;q)_i(q^2;q^2)_j} x^i y^{2j} = (y;q)_\infty \sum_{j\ge 0} \frac{(-x/y;q)_j}{(q;q)_j(y;q)_j} q^{\binom{j}{2}} y^j.
\end{align*}
By a basic algebraic computation, one can obtain the case \(E=0\) in \eqref{sec5-reduced}. For the detailed computation, one may refer to \cite[Thm.~3.8]{LW24}. Thus, building on our previous framework, the combinatorial proof of Theorem \ref{thm:5} is complete.

It is worth noting that the involution $\varphi^{\V}$ in Lemma \ref{lem:sec5-varphiV}, when specialized to $k=0$, yields the following identity with parameter $x$:
\begin{align*}
\sum_{i, j\geq 0}\frac{(-1)^jx^{i+2j}q^{i^2/2+3j^2+2ij+i/2}}{(q; q)_i(q^2; q^2)_j}=\sum_{i\geq 0}\frac{x^iq^{i^2}}{(q; q)_i}.
\end{align*}
Setting $x=1$ and $x=q$, respectively, and applying the two classical Rogers-Ramanujan identities, we obtain two double-sum Rogers-Ramanujan type identities.

\section{Combinatorial proof of Theorem \ref{thm:9}}\label{sec:6}

Recall that $\sR_{i,k}$ is defined in Section \ref{sec:1}. We set $\sR_i:=\sR_{i,0}$ and $\sR:=\bigcup_{i\geq 0}\sR_i$. Thus $\sR$ is the set of partitions whose adjacent parts differ by at least $2$.
For the right-hand side of \eqref{eq9}, we have
\begin{align}\label{sec7-1}
\sum_{(\al,\pi)\in\sR\times\sD}
x^{\ell(\pi)}y^{\ell(\al)}q^{|\al|+|\pi|}
=\text{RHS of }\eqref{eq9}
=(-xq;q)_{\infty}\sum_{n\geq 0}\frac{y^nq^{n^2}}{(q;q)_n}.
\end{align}
We now consider the series on the left-hand side of \eqref{eq9}, which can be written as
\begin{align}\label{sec7-2}
\text{LHS of }\eqref{eq9}
=\sum_{i,j\geq 0}
\frac{x^iy^{i+j}q^{3i^2/2+j^2+2ij+i/2}}
{(q;q)_i(q;q)_j}
\sum_{k\geq 0}
\frac{x^kq^{(i+j+1)+\cdots+(i+j+k)}}{(q;q)_k}.
\end{align}

First, we interpret the inner sum. For fixed $i,j\geq 0$, let $\sB_{i,j}^{\VI}$ denote the set of strict partitions $\mu$ whose parts are labeled $x$ and whose smallest part, if any, is at least $i+j+1$. Then
\begin{align}\label{sec7-3}
\sum_{\mu\in\sB_{i,j}^{\VI}}x^{\ell(\mu)}q^{|\mu|}
=\sum_{k\geq 0}
\frac{x^kq^{(i+j+1)+\cdots+(i+j+k)}}{(q;q)_k}.
\end{align}

Next, we consider the outer sum of \eqref{sec7-2}. Define the base partition
\begin{align*}
\beta^{\VI,(i,j)}
:=1_y+3_y+\cdots+(2j-1)_y
 +(2j+2)_{xy}+(2j+5)_{xy}+\cdots+(2j+3i-1)_{xy},
\end{align*}
where a bracket is empty when its corresponding parameter is zero. Since $|\beta^{\VI,(i,j)}| = 3i^2/2+j^2+2ij+i/2$, we have
\begin{align}\label{sec7-4}
\sum_{(\beta^{\VI,(i,j)},\ga^1,\ga^2)\in\{\beta^{\VI,(i,j)}\}\times\sP_i\times\sP_j}
x^iy^{i+j}q^{|\beta^{\VI,(i,j)}|+|\ga^1|+|\ga^2|}
=
\frac{x^iy^{i+j}q^{3i^2/2+j^2+2ij+i/2}}
{(q;q)_i(q;q)_j}.
\end{align}

Let $\sA_{i,j}^{\VI}$ denote the set of partitions $\la\in\sR$ with $i$ parts labeled $xy$ and $j$ parts labeled $y$, subject to the following additional condition: if $\la_m$ is labeled $xy$, then $\la_m-\la_{m-1}\geq 3$ when $m>1$, and $\la_1\geq 2$ when $m=1$ (formally, $\la_m - \la_{m-1} \ge 3$ with $\la_0 = -1$). One can easily check that $\beta^{\VI, (i, j)}$ is the partition with the smallest weight in $\sA^{\VI}_{i, j}$. We obtain the following bijection.

\begin{lemma}\label{lem:sec7-psiVI}
For $i,j\geq 0$, there exists a bijection
\begin{align*}
\psi^{\VI}:
\{\beta^{\VI,(i,j)}\}\times\sP_i\times\sP_j
&\ri\sA_{i,j}^{\VI}\\
(\beta^{\VI,(i,j)},\ga^1,\ga^2)&\mapsto\la,
\end{align*}
such that $|\beta^{\VI,(i,j)}|+|\ga^1|+|\ga^2|=|\la|$, $\ell(\beta^{\VI,(i,j)})=\ell(\la)$, and $\ell_y(\beta^{\VI,(i,j)})=\ell_y(\la)$.
\end{lemma}

\begin{proof}
The proof proceeds analogously to that of Lemma \ref{lem:sec2-psiII}. The partition $\ga^1\in\sP_i$ acts on the $i$ parts labeled $xy$: since consecutive $xy$ parts in $\beta^{\VI,(i,j)}$ differ by at least $3$, adding the parts of $\ga^1$ maintains differences of at least $3$ between consecutive $xy$ parts. The partition $\ga^2\in\sP_j$ is incorporated by successive forward moves on the $y$ parts.

Specifically, a forward move increases a part $t_y$ by $1$. When a forward move of $t_y$ would result in a difference of less than $3$ with a subsequent part $(t+3)_{xy}$, an adjustment is made according to the rule:
\begin{align*}
t_y,\ (t+3)_{xy} \longmapsto (t+1)_{xy},\ (t+3)_y.
\end{align*}
This adjustment increases the total weight by $1$, swaps the relative ordering of the two labels, and preserves all difference conditions: the new part $(t+1)_{xy}$ differs from the preceding part by at least $3$ (since the former difference was at least $2$), and the new part $(t+3)_y$ differs from $(t+1)_{xy}$ by $2$, which is admissible for a $y$ part. The remaining moves, including the treatment of any resulting interaction with neighboring parts, follow the procedure in Lemma \ref{lem:sec2-psiII}.

The inverse map is obtained by backward moves and backward adjustments: a backward move replaces $t_y$ by $(t-1)_y$, and whenever a configuration $(t+1)_{xy}, (t+3)_y$ is encountered, the backward adjustment reverses the move:
\begin{align*}
(t+1)_{xy},\ (t+3)_y \longmapsto t_y,\ (t+3)_{xy},
\end{align*}
followed by the standard normalization described in Lemma \ref{lem:sec2-psiII}. This establishes the bijection $\psi^{\VI}$, preserving total weight and label lengths.
\end{proof}

Using \eqref{sec7-4} and Lemma \ref{lem:sec7-psiVI}, we obtain
\begin{align*}
\sum_{\la\in\sA_{i,j}^{\VI}}
x^{\ell_{xy}(\la)}y^{\ell(\la)}q^{|\la|}
=
\frac{x^iy^{i+j}q^{3i^2/2+j^2+2ij+i/2}}
{(q;q)_i(q;q)_j}.
\end{align*}
Combining this identity with \eqref{sec7-3} gives
\begin{align}\label{sec7-5}
\sum_{i,j\geq 0}
\sum_{(\la,\mu)\in\sA_{i,j}^{\VI}\times\sB_{i,j}^{\VI}}
x^{\ell_{xy}(\la)+\ell(\mu)}
y^{\ell(\la)}q^{|\la|+|\mu|}
=\text{LHS of }\eqref{eq9}.
\end{align}

Our next step is to construct a bijection from $\bigcup_{i,j\geq 0}(\sA_{i,j}^{\VI}\times\sB_{i,j}^{\VI})$ to $\sR\times\sD$, thereby completing the combinatorial proof of Theorem \ref{thm:9}.

\begin{lemma}\label{lem:sec7-varphiVI}
There exists a bijection
\begin{align*}
\varphi^{\VI}:
\bigcup_{i,j\geq 0}
(\sA_{i,j}^{\VI}\times\sB_{i,j}^{\VI})
&\ri\sR\times\sD\\
(\la,\mu)&\mapsto(\al,\pi),
\end{align*}
such that $|\la|+|\mu|=|\al|+|\pi|$, $\ell_y(\la)=\ell(\al)$, and $\ell(\mu)+\ell_{xy}(\la)=\ell(\pi)$.
\end{lemma}

\begin{proof}
Let $(\la,\mu)\in\sA_{i,j}^{\VI}\times\sB_{i,j}^{\VI}$, with $n = \ell(\la) = i+j$. If $n=0$, then $\la$ and $\al$ are empty, $i=0$, and $\pi = \mu$, so the statement holds trivially.

Now assume $n\ge 1$. Write $\la=\la_1 + \la_2 + \cdots + \la_n$, and let the $i$ parts labeled $xy$ occur at indices
\[
1 \le a_1 < a_2 < \cdots < a_i \le n.
\]
For each $1\leq m\leq i$, define
\[
c_m := n - a_m + 1.
\]
Since $1\leq a_1 < a_2 < \cdots < a_i \le n$, the integers $c_m$ are distinct and satisfy
\[
n \ge c_1 > c_2 > \cdots > c_i \ge 1.
\]
On the other hand, by definition of $\sB_{i,j}^{\VI}$, the partition $\mu$ is strict and all of its parts satisfy $\mu_r \geq n+1$. Therefore, the union $\{c_1, \ldots, c_i\} \cup \{\mu_1, \ldots, \mu_{\ell(\mu)}\}$ consists of distinct positive integers. Arranging these parts in strictly increasing order defines a strict partition $\pi\in\sD$ of length
\[
\ell(\pi) = i + \ell(\mu) = \ell_{xy}(\la) + \ell(\mu).
\]

Next, define $\al = (\al_1, \ldots, \al_n)$ by setting, for each $1\leq r\leq n$,
\[
\al_r := \la_r - \#\{m : a_m \leq r\}.
\]
We verify that $\al\in\sR$:
\begin{enumerate}
\item For the first part $\al_1$: if $\la_1$ is labeled $xy$, then $\la_1 \ge 2$ by definition of $\sA_{i,j}^{\VI}$, and hence $\al_1 = \la_1 - 1 \ge 1$. If $\la_1$ is labeled $y$, then $\al_1 = \la_1 \ge 1$. Thus $\al_1 \ge 1$ in all cases.
\item For consecutive parts: for each $2\leq r\leq n$,
\[
\al_r - \al_{r-1} = (\la_r - \la_{r-1}) - \big(\#\{m : a_m \leq r\} - \#\{m : a_m \leq r-1\}\big).
\]
If $r = a_m$ for some $m$ (so $\la_r$ is labeled $xy$), then $\la_r - \la_{r-1} \ge 3$, while the difference between the two counts is $1$. Thus $\al_r - \al_{r-1} \ge 3 - 1 = 2$.
If $r \notin \{a_1, \ldots, a_i\}$ (so $\la_r$ is labeled $y$), then $\la_r - \la_{r-1} \ge 2$, while the difference between the two counts is $0$. Thus $\al_r - \al_{r-1} \ge 2 - 0 = 2$.
\end{enumerate}
Therefore, $\al$ is a partition whose consecutive parts differ by at least $2$, so $\al\in\sR$, with length $\ell(\al) = n = \ell_y(\la)$.

For the weight conservation, observe that subtracting $1$ from each part $\la_r$ for $r \ge a_m$ decreases the sum of parts of $\la$ by $n - a_m + 1 = c_m$. Summing over all $1\le m\le i$, we obtain
\[
|\al| = |\la| - \sum_{m=1}^i c_m.
\]
Since the parts of $\pi$ are the disjoint union of $\{c_1, \ldots, c_i\}$ and the parts of $\mu$, we have $|\pi| = |\mu| + \sum_{m=1}^i c_m$. Adding these two equalities yields
\[
|\al| + |\pi| = |\la| + |\mu|.
\]

To show that $\varphi^{\VI}$ is a bijection, we construct its inverse map. Let $(\al, \pi) \in \sR \times \sD$, and set $n = \ell(\al)$. If $n=0$, then $\la$ is empty and $\mu = \pi \in \sB_{0,0}^{\VI}$. For $n \ge 1$, split $\pi$ into the parts at most $n$ and the parts strictly greater than $n$:
\[
\{c \in \pi : c \le n\} = \{c_1 > c_2 > \cdots > c_i\}, \qquad \mu = \{c \in \pi : c \ge n+1\}.
\]
Since $\pi\in\sD$, the parts of $\mu$ are distinct and all at least $n+1 = i+j+1$, where $j := n - i \ge 0$, so $\mu \in \sB_{i,j}^{\VI}$. For each $1\le m\le i$, define $a_m := n - c_m + 1$. Because $n \ge c_1 > \cdots > c_i \ge 1$, we have $1 \le a_1 < \cdots < a_i \le n$.
We then define $\la = \la_1 + \cdots + \la_n$ by
\[
\la_r := \al_r + \#\{m : a_m \le r\},
\]
and label $\la_r$ with $xy$ if $r \in \{a_1, \dots, a_i\}$, and with $y$ otherwise.
Since $\al\in\sR$, $\al_r - \al_{r-1} \ge 2$. Thus $\la_r - \la_{r-1} \ge 2$ when $\la_r$ is labeled $y$, and $\la_r - \la_{r-1} \ge 2 + 1 = 3$ when $\la_r$ is labeled $xy$. If $\la_1$ is labeled $xy$, then $\la_1 = \al_1 + 1 \ge 2$. Hence $\la \in \sA_{i,j}^{\VI}$.
These two maps are clearly mutually inverse, establishing that $\varphi^{\VI}$ is a well-defined bijection.
\end{proof}

\begin{example}
In this example, we give the complete one-to-one correspondence under the bijection \(\varphi^{\VI}\) at weight $5$, as shown in Table \ref{tab:varphiVI}.

\begin{table}[htbp]
\centering
\renewcommand{\arraystretch}{1.25}
\begin{tabular}{ccc|ccc}
\toprule
 $(\lambda, \mu)$ & & $(\al, \pi)$ & $(\lambda, \mu) $ & & $(\al, \pi)$\\
\midrule
 $(\epsilon, 1+4)$         & $\longleftrightarrow$ & $(\ep, 1+4)$ & $(\ep, 2+3)$                & $\longleftrightarrow$ & $(\ep, 2+3)$  \\
 $(\ep, 5)$               & $\longleftrightarrow$ & $(\ep, 5)$ & $(1_y, 4)$           & $\longleftrightarrow$ & $(1, 4)$ \\
 $(2_{xy}, 3)$               & $\longleftrightarrow$ & $(1, 1+3)$ & $(2_y, 3)$ & $\longleftrightarrow$ & $(2, 3)$  \\
 $(3_{xy}, 2)$               & $\longleftrightarrow$ & $(2, 1+2)$ & $(3_y, 2)$            & $\longleftrightarrow$ & $(3, 2)$  \\
 $(5_{xy}, \ep)$                & $\longleftrightarrow$ & $(4, 1)$ & $(5_y, \ep)$            & $\longleftrightarrow$ & $(5, \ep)$  \\
 $(1_y+4_{xy},\ep)$                & $\longleftrightarrow$ & $(1+3, 1)$  & $(1_y+4_y, \ep)$            & $\longleftrightarrow$ & $(1+4, \ep)$  \\
\bottomrule
\end{tabular}
\vspace{0.5em}
\caption{The bijection $\varphi^{\mathrm{VI}}$ for total weight $5$.}
\label{tab:varphiVI}
\end{table}
\end{example}

By Lemma \ref{lem:sec7-varphiVI}, the sum on the left-hand side of \eqref{sec7-5} becomes
\begin{align*}
\sum_{(\al,\pi)\in\sR\times\sD}
x^{\ell(\pi)}y^{\ell(\al)}q^{|\al|+|\pi|}.
\end{align*}
This coincides with the generating function on the left-hand side of \eqref{sec7-1}. Since the right-hand sides of \eqref{sec7-5} and \eqref{sec7-1} are the left-hand and right-hand sides of \eqref{eq9}, respectively, we conclude that both sides of \eqref{eq9} are equal, which completes the combinatorial proof of Theorem \ref{thm:9}.

\section{Combinatorial proof of Theorem \ref{thm:8}}\label{sec:7}

In this section, first we investigate the combinatorial decomposition of the set 
$\sD$ of strict partitions. We define a deterministic multi-stage reduction algorithm that extracts difference-$t$ pairs 
for $t = 1, \dots, m$, establish its well-definedness, and characterize the combinatorial decomposition induced by it. Then we give a combinatorial proof of Theorem \ref{thm:8} by proving a more general multi-parameter generating function.

\subsection{Multi-stage peeling algorithm and combinatorial decomposition}
\label{subsec:6-1}
First we fix an integer $m \ge 0$ representing the peeling depth. When $m = 0$, no stages are 
performed. We now give an iterative definition of the algorithm.

\begin{Def}\label{def:6-1}
Let $m \ge 1$. For an initial partition $\lambda \in \sD$, initialize $\lambda^{(0)} = \lambda$. 
For each successive stage $t = 1, 2, \dots, m$, the partition $\lambda^{(t)}$ is constructed from 
$\lambda^{(t-1)}$ by the following three steps:
\begin{description}
    \item[Maximal difference-$t$ sequence decomposition] 
    Decompose $\lambda^{(t-1)}$ into maximal contiguous sequences 
    $B = (c_1, c_2, \dots, c_r)$ such that $c_{j+1} - c_j = t$ for all $1 \le j < r$.
    
    \item[Left-to-right pairing] 
    Within each sequence $B = (c_1, \dots, c_r)$, form disjoint consecutive pairs from left to right:
    \[
    [c_1, c_2]_t, \; [c_3, c_4]_t, \; \dots, \; [c_{2p-1}, c_{2p}]_t, \quad \text{where } p = \lfloor r/2 \rfloor.
    \]
    These $p$ pairs are designated as peeled at stage $t$. If the sequence length $r = 2p+1$ is odd, the final part $c_r$ is designated as a surviving part. If $r = 2p$ is even, all parts in $B$ are peeled. Let $\pr_t(\lambda)$ denote the total number of pairs peeled across all sequences at stage $t$.

    \item[Dynamic right-pair shift] 
    For each surviving part $a$, let $R_t(a)$ denote the number of pairs peeled at stage $t$ 
    that lie strictly to the right of $a$ in $\lambda^{(t-1)}$. Define the shifted part
    \begin{equation}
    \label{eq:shift_rule}
    \tau_t(a) = a + 2t R_t(a).
    \end{equation}
    The sequence of shifted surviving parts, ordered strictly increasingly, forms the partition $\lambda^{(t)}$.
\end{description}
After $m$ stages, the surviving parts of $\lambda^{(m)}$ are designated as singletons. 
We define the singleton statistic $\sn_m(\lambda) = \ell(\lambda^{(m)})$ and the multi-stage peeling signature
\begin{equation}
\sigma_m(\lambda) = \bigl(\sn_m(\lambda); \, \pr_1(\lambda), \pr_2(\lambda), \dots, \pr_m(\lambda)\bigr) \in \mathbb{N}_0^{m+1}.
\end{equation}
\end{Def}

Now we give an example with \(m=3\) to fully appreciate the flow of this algorithm.
\begin{example}\label{ex:4232}
Let $m=3$, and consider the strict partition
\[
\lambda = \lambda^{(0)} = (1, 5, 6, 7, 8, 10, 12, 14, 16, 18, 20, 22, 25, 28, 31, 34, 38, 42) \in \sD,
\]
with total length $\ell(\lambda) = 18$.

\begin{description}
\item[Stage 1 ($t=1$)]
In $\la^{(0)}$, the unique maximal difference-$1$ sequence of length greater than $1$ is $(5, 6, 7, 8)$, pairing within $(5, 6, 7, 8)$ yields $\pr_1(\la) = 2$ pairs: $$[5, 6]_1 \quad \text{and} \quad [7, 8]_1.$$ Because the sequence has even length $4$, neither element survives from it.
The surviving set is $\mathcal{S}^{(1)} = \{1, 10, 12, 14, 16, 18, 20, 22, 25, 28, 31, 34, 38, 42\}$.
Both peeled pairs lie to the right of $1$, so $R_1(1) = 2$, while $R_1(a) = 0$ for all $a \ge 10$. Applying the dynamic shift $\tau_1(a) = a + 2 R_1(a)$, the resulting partition is
\[
\lambda^{(1)} = (5, 10, 12, 14, 16, 18, 20, 22, 25, 28, 31, 34, 38, 42),
\]
which has $\ell(\lambda^{(1)}) = 14$ parts and satisfies $\lambda^{(1)}_{j+1} - \lambda^{(1)}_j \ge 2$ for all $j$.

\item[Stage 2 ($t=2$)]
In $\lambda^{(1)}$, the unique maximal difference-$2$ sequence of length greater than $1$ is $(10, 12, 14, 16, 18, 20, 22)$, which has odd length $7$. Pairing from left to right yields $\pr_2(\la) = 3$ pairs:
\[
[10, 12]_2, \quad [14, 16]_2, \quad [18, 20]_2,
\]
leaving the final element $22$ as a survivor. Then we have the surviving set is $\mathcal{S}^{(2)} = \{5, 22, 25, 28, 31, 34, 38, 42\}$. All three peeled pairs lie to the right of $5$, giving $R_2(5) = 3$, whereas $R_2(a) = 0$ for all $a \ge 22$.
Applying $\tau_2(a) = a + 4 R_2(a)$, the resulting partition is
\[
\lambda^{(2)} = (17, 22, 25, 28, 31, 34, 38, 42),
\]
which has $\ell(\lambda^{(2)}) = 8$ parts and satisfies $\lambda^{(2)}_{j+1} - \lambda^{(2)}_j \ge 3$ for all $j$.

\item[Stage 3 ($t=3$)]
In $\lambda^{(2)}$, the unique maximal difference-$3$ sequence of length greater than $1$ is $(22, 25, 28, 31, 34)$, of odd length $5$.
Pairing yields $\pr_3(\la) = 2$ pairs:
\[
(22, 25) \quad \text{and} \quad (28, 31),
\]
leaving $34$ as a survivor.
The surviving set is $\mathcal{S}^{(3)} = \{17, 34, 38, 42\}$.
Both peeled pairs lie to the right of $17$, so $R_3(17) = 2$, while $R_3(a) = 0$ for all $a \geq 34$. Applying $\tau_3(a) = a + 6 R_3(a)$, the final partition is
\[
\lambda^{(3)} = (29, 34, 38, 42),
\]
which has $\sn_3(\la)=\ell(\lambda^{(3)}) = 4$ parts and satisfies $\lambda^{(3)}_{j+1} - \lambda^{(3)}_j \ge 4$ for all $j$.

\item[Conclusion]
The multi-stage peeling algorithm terminates with signature
\[
\sigma_3(\lambda) = (\sn_3(\la); \pr_1(\la), \pr_2(\la), \pr_3(\la)) = (4; 2, 3, 2).
\]
The length conservation identity is verified:
\[
\ell(\lambda) = 18 = 4 + 2(2 + 3 + 2) = \sn_3(\la)+2(\pr_1(\la)+ \pr_2(\la)+\pr_3(\la)).\]
\end{description}
\end{example}

We also need to prove the following well-definedness of this algorithm.
\begin{lemma}\label{lem:gap_preservation}
Let $t \ge 1$. In the algorithm of Definition \ref{def:6-1}, if $\lambda^{(t-1)}$ consists of strictly positive integers whose adjacent parts 
differ by at least $t$, then $\lambda^{(t)}$ consists of strictly positive integers whose adjacent parts differ by at least $t+1$. Consequently, for any $\lambda \in \sD$ and any finite $m \ge 0$, the algorithm is deterministic and strictly well-defined.
\end{lemma}

\begin{proof}
Positivity is immediate: for any surviving part $a \ge 1$, since $R_t(a) \ge 0$, we have 
$\tau_t(a) = a + 2t R_t(a) \ge a \ge 1$.

If $\lambda^{(t)}$ contains zero or one part, the gap condition holds vacuously. 
Assume $\lambda^{(t)}$ contains at least two parts, and let $a < b$ be two consecutive survivors in $\lambda^{(t-1)}$. Every surviving part is the terminal element of an odd-length maximal difference-$t$ sequence. 
Because a maximal sequence contains at most one terminal part, $a$ and $b$ necessarily belong to distinct sequences.

Let $r$ denote the number of pairs peeled at stage $t$ that lie strictly between $a$ and $b$. 
Then exactly $2r$ parts of $\lambda^{(t-1)}$ lie strictly between $a$ and $b$. 
Tracing adjacent parts in $\lambda^{(t-1)}$ from $a$ to $b$ traverses $2r+1$ consecutive steps:
\[
a = w_0 < w_1 < w_2 < \dots < w_{2r} < w_{2r+1} = b.
\]
By hypothesis, $w_{j+1} - w_j \ge t$ for all $0 \le j \le 2r$. 
Moreover, because $a = w_0$ is the terminal part of a maximal difference-$t$ sequence, the very first step $w_1 - w_0$ crosses the boundary of that sequence, which implies $w_1 - w_0 \ge t+1$. Therefore,
\begin{equation}
\label{eq:raw_gap_bound}
b - a = (w_1 - w_0) + \sum_{j=1}^{2r} (w_{j+1} - w_j) \ge  2tr + t + 1.
\end{equation}
This bound holds for all $r \ge 0$.

Next, by definition of the right-pair counting function $R_t(\cdot)$, any peeled pair strictly 
to the right of $a$ that is not strictly to the right of $b$ must lie strictly between $a$ and $b$. 
Hence, $R_t(a) - R_t(b) = r$. Applying the shift rule~\eqref{eq:shift_rule}, we obtain
\begin{align*}
\tau_t(b) - \tau_t(a) 
&= \bigl(b + 2t R_t(b)\bigr) - \bigl(a + 2t R_t(a)\bigr) = (b - a) - 2tr.
\end{align*}
Substituting the lower bound~\eqref{eq:raw_gap_bound} yields
\[
\tau_t(b) - \tau_t(a) \ge (2tr + t + 1) - 2tr = t+1.
\]
Since the initial partition $\lambda^{(0)} = \lambda \in \sD$ satisfies $\lambda_{j+1} - \lambda_j \ge 1$, mathematical induction on $t$ guarantees that all parts of $\lambda^{(t)}$ differ by at least $t+1$. Strict positivity and strict monotonicity are preserved throughout, establishing well-definedness.
\end{proof}

\begin{theorem}\label{thm:length_conservation}
Let $\mathbf{k}=(k_1, k_2, ..., k_m)\in \bN_0^m$. For any $\lambda \in \sD$ and peeling depth $m \ge 0$, we have
\begin{equation}
\label{eq:length_identity}
\ell(\lambda) = \sn_m(\lambda) + 2\sum_{t=1}^m \pr_t(\lambda).
\end{equation}
Consequently, defining
\begin{equation}
\sD_{i, \mathbf{k}} = \bigl\{\lambda \in \sD : \sigma_m(\lambda) = (i; k_1, \dots, k_m)\text{ and }\ell(\la)=i+2\sum_{t=1}^mk_t\bigr\},
\end{equation}
the partition family $\{\sD_{i, \mathbf{k}}\}_{(i, \mathbf{k}) \in \mathbb{N}_0^{m+1}}$ forms a 
disjoint decomposition of $\sD$:
\begin{equation}
\sD = \bigcup_{(i, \mathbf{k} )\in \mathbb{N}_0^{m+1}} \sD_{i, \mathbf{k}}.
\end{equation}
\end{theorem}

\begin{proof}
At stage $t$, each of the $\pr_t(\lambda)$ peeled pairs removes exactly two parts. 
The shift $\tau_t$ is a strictly order-preserving injection on the surviving parts, so it preserves their count. 
Thus, $\ell(\lambda^{(t)}) = \ell(\lambda^{(t-1)}) - 2\pr_t(\lambda)$. 
Summing this recurrence from $t=1$ to $m$ gives
\[
\ell(\lambda^{(m)}) = \ell(\lambda^{(0)}) - 2\sum_{t=1}^m \pr_t(\lambda).
\]
Substituting $\ell(\lambda^{(0)}) = \ell(\lambda)$ and $\ell(\lambda^{(m)}) = \sn_m(\lambda)$ yields~\eqref{eq:length_identity}. 
The disjointness of the sets follows from the uniqueness and determinism of $\sigma_m(\lambda)$.
\end{proof}

Now for $(i, \mathbf{k})\in\bN_0^{m+1}$ with $m\geq 0$ and $K_t=\sum_{s=1}^{t-1}2sk_s$, we define the partition
\begin{align}
\beta^{(i, \mathbf{k})}=\beta^{(k_1)}\oplus\beta^{(k_2)}\oplus\cdots\oplus \beta^{(k_m)}\oplus\beta^{(i)},
\end{align}
where for $1\leq t\leq m$
\begin{align}
\beta^{(k_t)}:=[K_t+1, K_t+t+1]_t,\cdots ,[K_{t+1}-2t+1, K_{t+1}-t+1]_t,
\end{align}
and
\begin{align}
\beta^{(i)}:=(K_{m+1}+1), (K_{m+1}+m+2), \dots,(K_{m+1}+(m+1)i-m).
\end{align}
Next, we will show that \(\beta^{(i, \mathbf{k})}\) is indeed the base partition of \(\sD_{i, \mathbf{k}}\), that is, \(\beta^{(i, \mathbf{k})}\) is the partition in \(\sD_{i, \mathbf{k}}\) with the smallest weight.

\begin{theorem}\label{thm:base_partition_minimality}
Let $m \ge 0$, for $i \in \mathbb{N}_0$, and $\mathbf{k} = (k_1, \dots, k_m) \in \mathbb{N}_0^m$, we have $\beta^{(i, \mathbf{k})} \in \sD_{i, \mathbf{k}}$. Moreover, for every $\lambda \in \sD_{i, \mathbf{k}}$,
\begin{equation}\label{eq:componentwise_dominance}
\lambda_j \ge \beta^{(i, \mathbf{k})}_j \qquad \text{for all } 1 \le j \le \ell(\lambda).
\end{equation}
Consequently, $\beta^{(i, \mathbf{k})}$ is the unique partition of the smallest weight in $\sD_{i, \mathbf{k}}$.
\end{theorem}

\begin{proof}
We first establish two fundamental lower bounds for a single peeling stage. 
Let $x = (x_1, \dots, x_{n+2k})$ be a partition whose adjacent parts differ by at least $t$. 
Suppose that stage $t$ peels $k$ pairs and yields the partition $y = (y_1, \dots, y_n)$ 
under the dynamic shift rule~\eqref{eq:shift_rule}.
 Suppose \(y\neq\ep\) (if \(y=\ep\), this indicates that the pairs have been completely peeled).
Let $p_1 < p_2 < \dots < p_n$ be the indices in $x$ of the survivors that produce $y_1, \dots, y_n$. 
If $q_h$ peeled pairs lie strictly to the left of $x_{p_h}$, then exactly $h-1$ survivors and 
$2q_h$ peeled parts precede $x_{p_h}$, so $p_h = h + 2q_h$. 
The remaining $k - q_h$ peeled pairs lie strictly to its right, giving $R_t(x_{p_h}) = k - q_h$. 
By~\eqref{eq:shift_rule},
\[
y_h = x_{p_h} + 2t(k - q_h) = x_{h+2q_h} + 2t(k - q_h).
\]
Because adjacent parts of $x$ differ by at least $t$, we have $x_v - x_u \ge (v - u)t$ for all $u < v$. 
This yields two direct consequences:
\begin{itemize}
    \item \textbf{Minimum part increase:} When $n \ge 1$, setting $h = 1$ gives
    \begin{equation}\label{eq:stage_min_bound}
    y_1 = x_{1+2q_1} + 2t(k - q_1) \ge (x_1 + 2t q_1) + 2t(k - q_1) = x_1 + 2tk.
    \end{equation}
    \item \textbf{Survivor dominance:} For each $1 \le h \le n$, since $q_h \le k$, we have $p_h \le 2k + h$. Thus
    \begin{equation}\label{eq:stage_survivor_bound}
    x_{2k+h} \ge x_{p_h} + ((2k + h) - p_h)t = x_{p_h} + 2t(k - q_h) = y_h.
    \end{equation}
\end{itemize}

Now let $\lambda \in \sD_{i, \mathbf{k}}$, and consider the sequence of intermediate partitions 
$\lambda^{(0)} = \lambda, \lambda^{(1)}, \dots, \lambda^{(m)}$ generated by the algorithm. 
Since $\lambda^{(0)}_1 \ge 1 = K_1 + 1$, successive application of~\eqref{eq:stage_min_bound} gives
\[
\lambda^{(t-1)}_1 \ge 1 + \sum_{s=1}^{t-1} 2sk_s = K_t + 1 \qquad (1 \le t \le m+1).
\]
By Lemma~\ref{lem:gap_preservation}, the parts of $\lambda^{(t-1)}$ differ by at least $t$. 
Therefore, its first $2k_t$ parts satisfy
\begin{equation}\label{eq:block_lower_bound}
\lambda^{(t-1)}_j \ge \lambda^{(t-1)}_1 + (j - 1)t \ge K_t + 1 + (j - 1)t = \beta^{(k_t)}_j \qquad (1 \le j \le 2k_t).
\end{equation}
Furthermore, applying~\eqref{eq:stage_survivor_bound} repeatedly across stages $1, 2, \dots, t-1$ shows that 
for any $1 \le j \le 2k_t$,
\[
\lambda_{2\sum_{s=1}^{t-1} k_s + j} \ge \lambda^{(1)}_{2\sum_{s=2}^{t-1} k_s + j} \ge \dots \ge \lambda^{(t-1)}_j \ge \beta^{(k_t)}_j.
\]
Similarly, for the final $i$ singleton parts, Lemma~\ref{lem:gap_preservation} ensures that adjacent parts 
of $\lambda^{(m)}$ differ by at least $m+1$, so for every $1 \le h \le i$,
\[
\lambda_{2\sum_{s=1}^m k_s + h} \ge \dots \ge \lambda^{(m)}_h \ge \lambda^{(m)}_1 + (h - 1)(m + 1) \ge K_{m+1} + 1 + (h - 1)(m + 1) = \beta^{(i)}_h.
\]
Combining these bounds establishes the componentwise dominance $\lambda_j \ge \beta^{(i, \mathbf{k})}_j$ for all $1 \le j \le \ell(\lambda)$.

To see that $\beta^{(i, \mathbf{k})} \in \sD_{i, \mathbf{k}}$, observe that at the start of stage $t$, 
the unpeeled parts consist of $\beta^{(k_t)} \oplus \dots \oplus \beta^{(k_m)} \oplus \beta^{(i)}$. 
If $\sum_{s=t+1}^mk_s+i>0$, then the block $\beta^{(k_t)}$ has $2k_t$ parts with successive difference $t$, and the difference from its 
last part $K_{t+1}-t+1$ to the next part $K_{t+1}+1$ is also $t$. 
All subsequent differences are at least $t+1$. 
Thus, the maximal difference-$t$ sequence starting at $K_t+1$ has odd length $2k_t+1$. 
Left-to-right pairing peels exactly the $k_t$ pairs of $\beta^{(k_t)}$ and leaves $K_{t+1}+1$ as a survivor. 
All peeled pairs lie strictly to the left of every survivor, so $R_t(a) = 0$ for all survivors, 
implying $\tau_t(a) = a$. 
If \(\sum_{s=t+1}^m k_s + i = 0\), then the partition ends with \(\beta^{(k_t)}\). This sequence has difference \(t\) and length exactly \(2k_t\). Left-to-right pairing completely peels the \(k_t\) pairs, leaving no survivor.
By induction on $t$, $\mathrm{pr}_t(\beta^{(i, \mathbf{k})}) = k_t$ for each $t$, and $\mathrm{sn}_m(\beta^{(i, \mathbf{k})}) = i$, 
confirming $\beta^{(i, \mathbf{k})} \in \sD_{i, \mathbf{k}}$.

Finally, summing~\eqref{eq:componentwise_dominance} over all $j$ yields $|\lambda| \ge |\beta^{(i, \mathbf{k})}|$. 
Equality holds if and only if $\lambda_j = \beta^{(i, \mathbf{k})}_j$ for all $j$, which uniquely identifies 
$\lambda = \beta^{(i, \mathbf{k})}$.
\end{proof}

\subsection{A refined bijection} 
In this subsection, we will prove the following refined bijection with multiple parameters, thereby completing the combinatorial proof of Theorem \ref{thm:8}. Recall that $\sP_n=\{\la\in \sP: \ell(\la)\leq n\}$ and $\sE_n=\{\la\in \sE: \ell(\la)\leq n\}$ where $\sE$ is the set of partitions into even parts.

\begin{lemma}\label{lem:sec6-varphiVI}
For $m\geq 0$, there exists a bijection
\begin{align*}
\varphi^{\VII}: \{\beta^{(i, \mathbf{k})}\}\times \sE_{k_1}\times \cdots \times \sE_{k_m}\times \sP_i&\ri\sD_{i, \mathbf{k}}\\
(\beta^{(i, \mathbf{k})}, \ga^1, ..., \ga^m, \ga^{m+1})&\mapsto \la,
\end{align*}
such that $|\beta^{(i, \mathbf{k})}|+\sum_{t=1}^{m+1}|\ga^t|=|\la|$, $\ell(\beta^{(i, \mathbf{k})})=\ell(\la)$, and $\si_m(\beta^{(i, \mathbf{k})})=\si_m(\la)$.\end{lemma}

Note that $|\beta^{(i, \mathbf{k})}|=M(i, \mathbf{k})$ defined in \eqref{M}. Therefore, once the proof of Lemma \ref{lem:sec6-varphiVI} is complete, we have
\begin{align}
\sum_{\la\in \sD_{i, \mathbf{k}}}\left(\prod_{t=1}^mx_t^{\pr_t(\la)}\right)y^{\ell(\la)}q^{|\la|}&=\sum_{(\beta^{(i, \mathbf{k})}, \ga^1, ..., \ga^m, \ga^{m+1})}\left(\prod_{t=1}^mx_t^{\pr_t(\beta^{(i, \mathbf{k})})}\right)y^{\ell(\beta^{(i, \mathbf{k})})}q^{|\beta^{(i, \mathbf{k})}|+\sum_{t=1}^{m+1}|\ga^t|}\nonumber\\
&=\frac{\left(\prod_{t=1}^mx_t^{k_t}\right)y^{i+\sum_{t=1}^m2k_t}q^{M(i, k_1, ..., k_m)}}{(q; q)_i(q^2; q^2)_{k_1}(q^2; q^2)_{k_2}\cdots (q^2; q^2)_{k_m}},\label{sec6-9}
\end{align}
where \((\beta^{(i, \mathbf{k})}, \ga^1, ..., \ga^m, \ga^{m+1})\) ranges over \(\{\beta^{(i, \mathbf{k})}\}\times \sE_{k_1}\times \cdots \times \sE_{k_m}\times \sP_i\). Then summing both sides of \eqref{sec6-9} over all \((i, \mathbf{k})\in \bN_0^{m+1}\) gives
\begin{align}
\sum_{\la\in \sD}\left(\prod_{t=1}^mx_t^{\pr_t(\la)}\right)y^{\ell(\la)}q^{|\la|}=\sum_{i, k_1, \cdots, k_m\geq 0}\frac{\left(\prod_{t=1}^mx_t^{k_t}\right)y^{i+\sum_{t=1}^m2k_t}q^{M(i, k_1, ..., k_m)}}{(q; q)_i(q^2; q^2)_{k_1}(q^2; q^2)_{k_2}\cdots (q^2; q^2)_{k_m}}.
\end{align}
Finally, setting \(x_t=1\) for all \(1\leq t\leq m\) and \(y\mapsto x\), we obtain Theorem \ref{thm:8}.

\begin{proof}[Proof of Lemma \ref{lem:sec6-varphiVI}]
We construct the forward mapping $\varphi^{\mathrm{VII}}$ and its inverse $(\varphi^{\mathrm{VII}})^{-1}$ via an explicit five-step procedure. Throughout the construction, incremental partitions are written with their components weakly increasing:
\[
\gamma^t = \gamma^t_1 + \gamma^t_2 + \dots + \gamma^t_{k_t} \in \sE_{k_t} \quad (1 \le t \le m), \qquad \gamma^{m+1} = \gamma^{m+1}_1 + \gamma^{m+1}_2 + \dots + \gamma^{m+1}_i \in \sP_i,
\]
where each $\gamma^t_r \in 2\mathbb{N}_0$ is a non-negative even integer, and each $\gamma^{m+1}_j \in \mathbb{N}_0$ is a non-negative integer.

\noindent\textbf{Base partition structure and parameter decomposition.}
Recall that $K_t = \sum_{s=1}^{t-1} 2s k_s$ for $1 \le t \le m+1$. The base partition 
\[
\beta^{(i, \mathbf{k})} = \beta^{(k_1)} \oplus \beta^{(k_2)} \oplus \dots \oplus \beta^{(k_m)} \oplus \beta^{(i)}
\]
consists of $k_t$ difference-$t$ pairs at stage $t$, denoted in increasing order by
\[
P_{t, 1} < P_{t, 2} < \dots < P_{t, k_t}, \qquad \text{where } P_{t, r} = [a_{t, r}, a_{t, r}+t]_t \text{ with } a_{t, r} = K_t + 1 + 2(r-1)t,
\]
and $i$ singletons denoted in increasing order by
\[
s_1 < s_2 < \dots < s_i, \qquad \text{where } s_j = K_{m+1} + 1 + (j-1)(m+1).
\]
By Theorem~\ref{thm:base_partition_minimality}, $\beta^{(i, \mathbf{k})} \in \sD_{i, \mathbf{k}}$, and peeling stage $t$ peels precisely the $k_t$ pairs $P_{t, 1}, \dots, P_{t, k_t}$ leaving all survivors strictly to their right.

Each singleton $s_j$ must absorb weight $\gamma^{m+1}_j$. Each pair $P_{t, r}$ must absorb an even weight $\gamma^t_r$; we decompose this into $\gamma^t_r / 2$ forward moves, each of which increases the sum of the partition by exactly $2$.

\noindent
\textbf{Step 2: Forward moves and adjustments.}
We define two types of elementary operations on parts of the partition:
\begin{description}
    \item[Singleton translation] For each $1 \le j \le i$, the singleton $s_j$ is shifted forward by $\gamma^{m+1}_j$:
    \[
    s_j \longmapsto s'_j = s_j + \gamma^{m+1}_j.
    \]
    Since $s_{j+1} - s_j = m+1$ and $\gamma^{m+1}_{j+1} \ge \gamma^{m+1}_j$, we have
    \[
    s'_{j+1} - s'_j = (s_{j+1} - s_j) + (\gamma^{m+1}_{j+1} - \gamma^{m+1}_j) \ge m+1.
    \]
    Thus, singletons remain strictly increasing with adjacent gaps at least $m+1$.
    
    \item[Pair forward moves and adjustments] Let $[a, a+t]_t$ be an active difference-$t$ pair at stage $t$. Let $w$ denote the immediately succeeding part in the current partition (if such a part exists). A forward step on $[a, a+t]_t$ is governed by the following dichotomy:
    \begin{itemize}
        \item \emph{Case I (Free forward move):} If $w \ge a+2t+1$, or if $[a, a+t]_t$ is the terminal block of the partition, the pair advances freely by $+1$ on each component:
        \[
        [a, a+t]_t \longmapsto [a+1, a+t+1]_t.
        \]
        The sum of parts increases by $2$. The adjacent difference to $w$ becomes $w - (a+t+1) \ge (a+2t+1) - (a+t+1) = t$, preserving the minimum gap $t$.
        
        \item \emph{Case II (Adjustment):} If $w = a+2t$, a forward move would produce $[a+1, a+t+1]_t$ followed by $a+2t$, which would create an adjacent difference of $(a+2t) - (a+t+1) = t-1 < t$. In this critical case, we perform the adjustment:
        \[
        [a, a+t]_t, \, a+2t \longmapsto a, \, [a+t+1, a+2t+1]_t.
        \]
        The pair leapfrogs over $w = a+2t$. The parts transform from $(a, a+t, a+2t)$ to $(a, a+t+1, a+2t+1)$. The sum of the parts increases from $3a+3t$ to $3a+3t+2$, adding exactly $2$ to the weight. The adjacent gaps become $(a+t+1) - a = t+1 > t$ and $(a+2t+1) - (a+t+1) = t$, so all adjacent differences remain at least $t$.
    \end{itemize}
\end{description}

\noindent
\textbf{Step 3: Global scheduling order.}
To ensure that moves never interfere with one another or create spurious differences at earlier peeling stages, operations are scheduled in the following strict order:
\begin{enumerate}
    \item \textbf{Stage priority:} We first shift all singletons (stage $m+1$). Next, we process pair stages in strictly descending order: $t = m, \, m-1, \, \dots, \, 1$.
    \item \textbf{Right-to-left within stage $t$:} Within stage $t$, pairs are processed strictly from right to left: $r = k_t, \, k_t-1, \, \dots, \, 1$. For pair $P_{t, r}$, we perform exactly $\gamma^t_r / 2$ elementary steps (each being either a free move or an adjustment as dictated by the neighbor $w$).
\end{enumerate}
Because $r$ runs from $k_t$ down to $1$, the space to the right of $P_{t, r}$ is cleared before $P_{t, r}$ begins moving. If $P_{t, r+1}$ or singletons originally abutted $P_{t, r}$, they have already advanced, preventing cascading collisions.

\noindent
\textbf{Step 4: Statistic invariance.}
We now verify that every elementary operation preserves the signature $\sigma_m(\lambda) = (i; k_1, \dots, k_m)$.
\begin{description}
    \item[Invariance at earlier stages $s < t$]
    In $\beta^{(i, \mathbf{k})}$, all parts in blocks $s > t$ have differences at least $t$. Because stages are executed in descending order, all adjacent differences among parts participating in stage $t$ are at least $t$. A free move keeps all differences at least $t$. An adjustment replaces $(a, a+t, a+2t)$ with differences $(t, t)$ by $(a, a+t+1, a+2t+1)$ with differences $(t+1, t)$. Since $t \ge s+1$, every adjacent difference is strictly greater than $s$. Thus, no difference-$s$ pair is ever formed during stage $t$.

    \item[Invariance at stage $t$ and preservation of dynamic shift]
    Under a free move $[a, a+t]_t \mapsto [a+1, a+t+1]_t$, the pair remains a difference-$t$ pair and is peeled at stage $t$.
    Under an adjustment, $(a, a+t, a+2t)$ is replaced by $(a, a+t+1, a+2t+1)$. In the original configuration, left-to-right pairing peels $[a, a+t]_t$, leaving $a+2t$ as a survivor. In the adjusted configuration, the difference between $a$ and $a+t+1$ is $t+1 > t$, so $a$ is isolated and survives, while $[a+t+1, a+2t+1]_t$ has difference $t$ and is peeled.
    
    We compare the shifted surviving values under the dynamic shift rule~\eqref{eq:shift_rule}: $\tau_t(x) = x + 2t R_t(x)$.
    Let $R \ge 0$ be the number of peeled pairs strictly to the right of $a+2t$.
    \begin{itemize}
        \item In $(a, a+t, a+2t)$, the survivor is $a+2t$. Since all $R$ pairs lie strictly to its right, $R_t(a+2t) = R$. Its shifted image in $\lambda^{(t)}$ is:
        \[
        \tau_t(a+2t) = a+2t + 2t R_t(a+2t) = a + 2t + 2tR.
        \]
        \item In $(a, a+t+1, a+2t+1)$, the survivor is $a$. The newly formed peeled pair $[a+t+1, a+2t+1]_t$ lies strictly to the right of $a$, in addition to the $R$ pairs further to the right. Thus $R_t(a) = R + 1$. Its shifted image in $\lambda^{(t)}$ is:
        \[
        \tau_t(a) = a + 2t R_t(a) = a + 2t(R + 1) = a + 2t + 2tR.
        \]
    \end{itemize}
    The surviving part mapped into $\lambda^{(t)}$ is identically $a + 2t + 2tR$ in both cases.
    
    \item[Invariance at subsequent stages $s > t$]
    Because the set of surviving values entering $\lambda^{(t)}$ is identical and their relative positions are unchanged, all subsequent stages $t+1, \dots, m$ produce identical pairings, pair counts $k_s$, and final singletons.
\end{description}
Therefore, $\sigma_m(\lambda) = \sigma_m(\beta^{(i, \mathbf{k})}) = (i; k_1, \dots, k_m)$ holds invariant throughout the algorithm.

\noindent
\textbf{Step 5: Canonical normalization and invertibility.}
We define the inverse map $(\varphi^{\mathrm{VII}})^{-1}: \sD_{i, \mathbf{k}} \to \{\beta^{(i, \mathbf{k})}\} \times \sE_{k_1} \times \dots \times \sE_{k_m} \times \sP_i$. Given $\lambda \in \sD_{i, \mathbf{k}}$, we reconstruct the parameter partitions and recover $\beta^{(i, \mathbf{k})}$ by reversing the forward moves stage by stage in ascending order $t = 1, 2, \dots, m$:
\begin{description}
    \item[Pair normalization at stage $t$] At stage $t$, the partition contains $k_t$ peeled difference-$t$ pairs. We process them from left to right: $r = 1, 2, \dots, k_t$. For a pair $[b, b+t]_t$, let $u$ be the immediately preceding part in the partition (if any).  
    \begin{itemize}
        \item \emph{Case I (Normalization):} If $u = b - t - 1$, the triple is $(b-t-1, b, b+t)$. We apply the inverse adjustment:
        \[
        b-t-1, \, [b, b+t]_t \longmapsto [b-t-1, b-1]_t, \, b+t-1.
        \]
        This decreases the sum of parts by $2$.
        \item \emph{Case II (Free backward move):} If $u \le b - t - 2$ or if no preceding part exists or $u$ is a component of a difference-$s$ pair for any $s\leq t$, we shift the pair backward by $-1$ on each element:
        \[
        [b, b+t]_t \longmapsto [b-1, b+t-1]_t.
        \]
        This decreases the sum of parts by $2$.
    \end{itemize}
    We repeatedly apply these elementary backward steps to pair $r$ until its first component reaches the canonical base position $a_{t, r} = K_t + 1 + 2(r-1)t$. The total number of backward steps performed on pair $r$ multiplied by $2$ uniquely determines $\gamma^t_r \in 2\mathbb{N}_0$. Processing from left to right guarantees $0 \le \gamma^t_1 \le \dots \le \gamma^t_{k_t}$, so $\gamma^t \in \sE_{k_t}$.
    
    \item[Singleton normalization] After all stages $t=1, \dots, m$ are normalized to the base blocks $\beta^{(k_1)} \oplus \dots \oplus \beta^{(k_m)}$, the remaining $i$ parts are the singletons $s'_1 < s'_2 < \dots < s'_i$. We define:
    \[
    \gamma^{m+1}_j = s'_j - s_j \qquad (1 \le j \le i).
    \]
    Since $s'_{j+1} - s'_j \ge m+1 = s_{j+1} - s_j$, we have $0 \le \gamma^{m+1}_1 \le \gamma^{m+1}_2 \le \dots \le \gamma^{m+1}_i$, yielding $\gamma^{m+1} \in \sP_i$.
\end{description}
Because each backward operation uniquely inverts the corresponding forward operation, $(\varphi^{\mathrm{VII}})^{-1}$ is deterministic, well-defined, and satisfies that $(\varphi^{\mathrm{VII}})^{-1}$ is the inverse map of $\varphi^{\mathrm{VII}}$. This proves that $\varphi^{\mathrm{VII}}$ is a weight-preserving, length-preserving bijection onto $\sD_{i, \mathbf{k}}$, as described in the lemma.
\end{proof}

Finally, we present an example to illustrate the bijection \(\varphi^{\VII}\) in Lemma \ref{lem:sec6-varphiVI}.

\begin{example}\label{ex:sec6-varphiVI}
Let $m=3$ and $(i;k_1,k_2,k_3)=(2;1,2,1)$. Choose
\[
\gamma^1=2,\qquad \gamma^2=2+8,\qquad \gamma^3=2,\qquad \gamma^4=1.
\]
The base partition, with its designated pairs indicated, is
\[
\beta^{(2;(1,2,1))} = ([1,2]_1, [3,5]_2, [7,9]_2, [11,14]_3, 17, 21).
\]
The singleton $21\mapsto 22$ since $\gamma^4=1$. At stage $3$, the difference-$3$ pair undergoes an adjustment:
\[
([1,2]_1,[3,5]_2,[7,9]_2,[11,14]_3,17,22)
\longmapsto
([1,2]_1,[3,5]_2,[7,9]_2,11,[15,18]_3,22).
\]
At stage $2$, we first process the rightmost difference-$2$ pair. Its four forward steps are
\begin{align*}
&([1,2]_1,[3,5]_2,[7,9]_2,11,15,18,22)\\
&\quad\longmapsto
 ([1,2]_1,[3,5]_2,7,[10,12]_2,15,18,22)\\
&\quad\longmapsto
 ([1,2]_1,[3,5]_2,7,[11,13]_2,15,18,22)\\
&\quad\longmapsto
 ([1,2]_1,[3,5]_2,7,11,[14,16]_2,18,22)\\
&\quad\longmapsto
 ([1,2]_1,[3,5]_2,7,11,14,[17,19]_2,22).
\end{align*}
The first, third, and fourth of these steps are adjustments. In particular, in the third and fourth steps, the moving difference-$2$ pair successively encounters the parts $15$ and $18$ resulting from the stage-$3$ pair, replacing each encountered part in turn via the adjustment rule.

Next, the remaining difference-$2$ pair and the difference-$1$ pair each undergo an adjustment:
\begin{align*}
&([1,2]_1,[3,5]_2,7,11,14,[17,19]_2,22)\\
&\quad\longmapsto
 ([1,2]_1,3,[6,8]_2,11,14,[17,19]_2,22)\\
&\quad\longmapsto
 (1,[3,4]_1,6,8,11,14,17,19,22).
\end{align*}
Consequently,
\[
\la:=\varphi^{\mathrm{VII}}\bigl(\beta^{(2;(1,2,1))}, 2, 2+8, 2, 1\bigr)
= (1,3,4,6,8,11,14,17,19,22)\in \sD_{2, 1, 2, 1}.
\]
The weight increases from $|\beta^{(2;(1,2,1))}|=90$ to $90+2+10+2+1=105$, matching the sum of the resulting partition. The verification of the inverse mapping is left to the reader.
\end{example}

\section{Conclusion}\label{sec:conclusion}

We have investigated several multi-sum Rogers-Ramanujan type identities through structural decompositions of strict partitions. The constructions established across the paper include weight-preserving bijections, parity-based stripping algorithms, and sign-reversing involutions. In particular, base+increments representations directly link the quadratic exponents of the series to minimal partition configurations and the $q$-shifted factorials in denominators to admissible increments. Two parameterized general families place various known and new specializations in a unified setting: one isolates a Rogers-Ramanujan partition from a strict partition, while the other organizes strict partitions according to multi-stage pairing signatures and yields a general $(m+1)$-fold multisum.

Although the explicit bijections take distinct forms according to the boundary conditions, they repeatedly exploit reversible local moves to rearrange parts, control structural partition statistics, and account exactly for changes in weight. A natural and compelling next step is to abstract these shared mechanisms into a unified, object-centred combinatorial framework rather than treating individual $q$-series identities in isolation. Such a theory would identify canonical base configurations, characterize compatible increment structures, and establish general criteria under which microscopic moves remain invertible and non-interfering. It should also naturally incorporate sign-reversing involutions whenever algebraic cancellation is required.

We anticipate that this object-centred perspective will provide partition-theoretic explanations for broader classes of Rogers-Ramanujan type multisums and help identify higher-rank Nahm-type identities that admit combinatorial models. Introducing refining parameters that track residue classes, pairing patterns, or transition dynamics may reveal deeper combinatorial symmetries. Beyond ad hoc sum-to-product evaluations, the overarching objective is to understand when and why divergent analytic identities emerge from the same underlying combinatorial geometry.








\end{document}